\documentclass{jag-l}
\usepackage{amsthm,amsmath,amssymb,amscd, mathrsfs}
\usepackage[latin1]{inputenc}
\usepackage[all]{xy}
\numberwithin{equation}{section}

 \newtheorem{thm}[equation]{Theorem}
 \newtheorem{prop}[equation]{Proposition}
 \newtheorem{lemma}[equation]{Lemma}
 \newtheorem{kor}[equation]{Corollary}
 \newtheorem{cor}[equation]{Corollary}

 \theoremstyle{definition}
 \newtheorem{definition}[equation]{Definition}
 \newtheorem{ex}[equation]{Example}
 \newtheorem{conj}[equation]{Conjecture}

 \theoremstyle{remark}
 \newtheorem{remark}[equation]{Remark}
 
 \def\vol{{\rm vol}}

\def\Hom{{\rm Hom}}
\def\rg{{\rm rk}}

 \def\ad{{\rm ad}}

 \def\det{{\rm det}}

 \def\inv{{\rm inv}}
 \def\dom{{\rm dom}}
 \def\gen{{\rm gen}}

\def\GL{{\rm GL}}

\def\O{{\mathcal{O}}}
\def\e{{\epsilon}}

\def\Z{\mathbb{Z}}
\def\Q{\mathbb{Q}}

\def\Res{{\rm Res}}

\def\Ind{{\rm Ind}}

\begin{document}
\begin{title}
{Affine Deligne-Lusztig varieties and the action of $J$}
\end{title}
\author{Miaofen Chen and Eva Viehmann}
\thanks{The first author was partially supported by NSFC grant No. 11301185, SRFDP grant No. 20130076120002, and STCSM grant No. 13dz2260400. The second author was partially supported by ERC starting grant 277889 ``Moduli spaces of local $G$-shtukas''.}
\date{}
\begin{abstract}{We propose a new stratification of the reduced subschemes of Rapoport-Zink spaces and of affine Deligne-Lusztig varieties that highlights the relation between the geometry of these spaces and the action of the associated automorphism group. We show that this provides a joint group-theoretic interpretation of well-known stratifications which  only exist for special cases such as the Bruhat-Tits stratification of Vollaard and Wedhorn, the semi-module stratification of de Jong and Oort, and the locus where the $a$-invariant is equal to $1$.}
\end{abstract}
\maketitle
\tableofcontents

\section{Introduction}\label{secintro}

Let $k$ be a finite field with $q=p^r$ elements and let $\overline{k}$ be an algebraic closure of $k$. We consider the arithmetic case where we set $F=W(\mathbb{F}_q)[1/p]$ and the function field case where $F=k((t))$. In both cases let $L$ denote the completion of the maximal unramified extension of $F$ and let $\mathcal{O}_F$ and $\mathcal{O}_L$ be the valuation rings. We denote by $\e$ the uniformizer $t$ or $p$. We write $\sigma:x\mapsto x^q$ for the Frobenius of $\overline{k}$ over $k$ and also for the induced Frobenius of $L$ over $F$ (mapping $\e$ to $\e$).

Let $G$ be a connected reductive group over $\mathcal{O}_F$ and let $K = G(\O_L).$ Since $k$ is finite $G$ is automatically quasi-split. Let $B \subset G$ be a Borel subgroup and $T \subset B$ a maximal torus in $B$, both defined over $\O_F$. We denote by $X_*(T)$ the set of cocharacters of $T,$ defined over $\O_L.$

We fix a minuscule dominant cocharacter $\mu\in X_*(T)$ and an element $b\in G(L)$. Then the affine Deligne-Lusztig variety $X^G_{\mu}(b)=X_{\mu}(b)$ is defined as follows. Consider the following set of points
\begin{equation*}
X_{\mu}(b)(\overline{k})=\{g\in G(L)/K\mid g^{-1}b\sigma(g)\in K\e^{\mu}K\}.
\end{equation*}
Here we use $\e^{\mu}:=\mu(\e)$. In the function field case, this can be seen as the set of $\overline{k}$-valued points of a reduced closed subscheme of the affine Grassmannian. In the arithmetic case it defines in the same way a closed subspace of the Witt vector affine Grassmannian (see \cite{ZhuGr}, and \cite{BS} for the scheme structure). In many cases (i.e.~when $(G,\mu)$ corresponds to a Shimura datum of Hodge type), this set is the set of $\overline{k}$-valued points of a Rapoport-Zink moduli space of $p$-divisible groups. Furthermore, the perfection of this Rapoport-Zink space is isomorphic to the affine Deligne-Lusztig variety equipped with the structure of a (perfect) subscheme of the affine Grassmannian. From now on we will use the name ``affine Deligne-Lusztig variety'' to refer to all of these cases.

Let $$J_b(F)=\{g\in G(L)\mid g\cdot b=b\cdot \sigma(g)\}.$$ This is the set of $F$-points of an algebraic group over $F$, an inner form of some Levi subgroup of $G$ (the centralizer of the Newton point $\nu_b$ of $b$, \cite{Ko1}). There is a natural action of $J_b(F)$ on $X_{\mu}(b)$.

The geometry of affine Deligne-Lusztig varieties has been studied by many people. For example we know about the sets of connected components of closed affine Deligne-Lusztig varieties (\cite{ckv}), and for minuscule $\mu$ and $G=\GL_n$ or $GSp_{2n}$ also their sets of irreducible components (\cite{modpdiv},\cite{polpdiv}). In several particular cases we even have a complete description of their geometry, for example Kaiser \cite{Kai} for the moduli space of supersingular $p$-divisible groups of dimension 2, Vollaard-Wedhorn \cite{VW} for certain unitary groups of signature $(1,n-1)$,   further generalized by G\"ortz and He in \cite{GH}. All of these results indicate a close relation between the geometry of the affine Deligne-Lusztig varieties and (the Bruhat-Tits building of) $J_b(F)$. However, so far a conceptual way to explain this is still lacking.

In this paper we propose a new invariant on affine Deligne-Lusztig varieties which also induces a decomposition of the variety into locally closed subschemes. Inspired by the decompositions for particular cases discussed above, our invariant has the property that it not only depends on the element $g^{-1}b\sigma(g)\in K\e^\mu K$, i.e.~on the $p$-divisible group or local $G$-shtuka at the point of the moduli space we are interested in, but also on the quasi-isogeny, resp.~on the element $g$ itself.

Based on the idea that the geometry of an affine Deligne-Lusztig variety should be studied in relation with the action of $J_b(F)$, we assign to an element $g\in G(L)/K$ the function
\begin{eqnarray*}
f_g:J_b(F)&\rightarrow&X_*(T)_{\dom}\\
j&\mapsto&\inv(j,g).
\end{eqnarray*}
Here $\inv$ denotes the relative position, i.e. the uniquely defined element of $X_*(T)_{\dom}$ with $j^{-1}g\in K\e^{f_g(j)}K$ given by the Cartan decomposition $$G(L)=\coprod_{\xi\in X_*(T)_{\dom}}K\xi(\e) K.$$ Note that this definition depends on the choice of $b$ within its $\sigma$-conjugacy class. For details on this dependence and on the choice of $b$ compare the beginning of Section \ref{sec2}.

We show in Section \ref{sec2} that this defines a decomposition of the affine Deligne-Lusztig variety into locally closed pieces. However, in general the closure of a stratum is not a union of strata (compare Section \ref{secclos} for a counterexample).

It turns out that this stratification is the natural group-theoretic generalization of a number of other stratifications that were studied intensively over the past years, but only existed for special cases, and were up to now unrelated to each other. We discuss three classes of such stratifications.\\

\noindent{\it 1. The Bruhat-Tits stratification. }In \cite{VW} Vollaard and Wedhorn consider the supersingular locus of Shimura varieties for unitary groups of signature $(1,n-1)$ at an inert prime. They show that a refinement of the Ekedahl-Oort stratification yields a stratification of this locus, such that the individual strata have a description in terms of fine Deligne-Lusztig varieties, and the closure relations are given in terms of the Bruhat-Tits building of the associated goup $J_b(F)$. Since then, this result has been prominently used, in particular in relation with the Kudla programme. In \cite{GH}, G\"ortz and He generalize this result by computing a complete list of cases of affine Deligne-Lusztig varieties for which the same generalization of the Ekedahl-Oort invariant yields an analogous result. We show in Section \ref{sec42} that in the Vollaard-Wedhorn case, our invariant coincides with theirs. We conjecture that the same holds in all cases studied by G\"ortz and He and verify this conjecture in one additional case. However, we also show that our invariant is not a refinement of the Ekedahl-Oort invariant in general. Based on our  conjecture and on the applications of Vollaard-Wedhorn's theory we expect that our stratification will prove to be an interesting tool to study further cases of basic loci of Shimura varieties.\\

\noindent{\it 2. Semi-modules.} The most successful approach to prove statements about affine Deligne-Lusztig varieties in the affine Grassmannian is to use group-theoretic methods to reduce the statement to the so-called superbasic case. It was, for example, employed to prove the dimension formula and to determine the sets of connected components. Proofs for the remaining superbasic case then  usually require very explicit computations. The main tool here is the stratification by semi-modules. It was first considered by de Jong and Oort in \cite{deJongOort} for certain moduli spaces of $p$-divisible groups corresponding to the group $\GL_n$ and later extended to the superbasic case for unramified groups in \cite{dimdlv} and \cite{hamacher}. We show in Section \ref{secsemmod} that it coincides with the special case for superbasic $b$ of our stratification.\\

\noindent{\it 3. The $a$-number. }Finally, we discuss the relation to the $a$-number of $p$-divisible groups. This invariant assigns with a $p$-divisible group $X$ over $k$ the natural number $\dim\Hom_k(\alpha_p,X)$. It is a particularly useful tool to study moduli spaces of $p$-divisible groups with or without polarization, but so far does not have a good generalization for $p$-divisible groups together with endomorphisms. It was used for example in Oort's proof of the Grothendieck conjecture on the closure of Newton strata. There he first deformed any (non-ordinary) $p$-divisible group with a given Newton polygon to one with the same Newton polygon but $a$-number 1, and then used explicit methods to compute their deformations. In general our invariant (for the group $GL_n$) seems to be not related to the $a$-number of $p$-divisible groups. However, in the crucial case of the generic $a$-number 1 (and basic $b$), we can show that this locus coincides with one $J_b(F)$-orbit of strata for our invariant. We conjecture that this result still holds if one drops the assumption that $b$ is basic. In this way, our invariant defines also a group-theoretic generalization of the open stratum defined by $a=1$.
\\

Altogether these examples show that the functions $f_g$ are an invariant that seems to be central in the study of the geometry of affine Deligne-Lusztig varieties. One might hope that it opens the way to a more systematic investigation of its relation to the Bruhat-Tits building of $J_b(F)$.

As one instance of this we want to consider in forthcoming work the case of the supersingular locus of Shimura varieties for $GU(n-2,2)$ at an inert prime, a case where (at least for $n>4$) the approach of Vollaard-Wedhorn or G\"ortz and He is no longer applicable, but where this new invariant leads again to a description of the corresponding affine Deligne-Lusztig variety, and thus also of the supersingular locus itself.\\

When we finished a preliminary version of this paper, X.~Zhu explained to us that in a preprint in preparation of him and L.~Xiao, they study the geometry of some affine Deligne-Lusztig varieties by intersecting with Schubert varieties in the affine Grassmannian with different focus than ours. More precisely, they describe the irreducible components of affine Deligne-Lusztig varieties when $b$ is of the form $\epsilon^\mu$. It would be interesting to study the relation between these two papers.

\noindent{\it Acknowledgement.} We thank B. Howard for helpful discussions and S. Orlik for telling us about the theory of thin Schubert cells. 

\section{Definition and basic properties}\label{sec2}

Let $G$ be an unramified reductive group over $\O_F$, and fix a Borel subgroup $B$ and a maximal torus $T$ of $G$, both defined over $\O_F$. Let $\mu\in X_*(T)$ be a dominant cocharacter. We fix an element $b\in G(L)$ and consider the affine Deligne-Lusztig variety
$$X_{\mu}(b)=\{g\in G(L)/K\mid g^{-1}b\sigma(g)\in K\e^{\mu}K\}.$$  Multiplication with an element $h\in G(L)$ induces an isomorphism $$X_{\mu}(b)\cong X_{\mu}(h^{-1}b\sigma(h)).$$ Therefore, the geometry of the affine Deligne-Lusztig variety only depends on the $\sigma$-conjugacy class $$[b]=\{h^{-1}b\sigma(h)\mid h\in G(L)\}.$$

As before we consider the group $$J_b(F)=\{g\in G(L)\mid g^{-1}b\sigma(g)=b\}$$ and its action on $X_{\mu}(b)$. We assign to each $g\in G(L)$ the function
\begin{eqnarray*}
f_g:J_b(F)&\rightarrow&X_*(T)_{\dom}\\
j&\mapsto&\inv(j,g).
\end{eqnarray*}
As it only depends on the $K$-coset of $g$, it induces a well defined invariant on $G(L)/K$, and on each $X_{\mu}(b)$. Furthermore, each $f_g$ is constant on cosets $j(K\cap J_b(F))$.

\begin{remark} Although the varieties $X_{\mu}(b)$ and $X_{\mu}(h^{-1}b\sigma(h))$ are isomorphic for every $h\in G(L)$, the invariants $f_g$ can lead to substantially different invariants on these varieties. To see this assume that $g\in X_{\mu}(b)$ for some $g$ and $b$. Then  under the above isomorphism for $h=g$, the element $g$ is mapped to $1\in X_{\mu}(g^{-1}b\sigma(g))$. To $1$ we now assign a function $$f_1:J_{g^{-1}b\sigma(g)}(F)\rightarrow X_*(T)_{\dom}$$ with the property that $f_1(1)$ is trivial. The element $1\in G(L)/K$ is the unique element with this property. On the other hand, there is no reason why the element $g$ should be the only one having given invariant $f_g:J_{b}(F)\rightarrow X_*(T)_{\dom}.$ This shows that it is important to choose the right representative $b\in [b]$ in order to obtain reasonable invariants on $X_{\mu}(b)$. Note that on the other hand, the stratifications according to our invariants do not change if we replace $b$ by a different representative within its $K$-$\sigma$-conjugacy class.

For the general results below we do not make any restriction on the representative $b$. However, in all of the examples or comparisons to other stratifications, we need to fix a specific representative of the given class. It turns out that in all of these cases it is reasonable to choose a representative of a so-called straight element $w_b\in\widetilde W$ in $G(L)$. These elements of the extended affine Weyl group are defined by the property that $l(w\sigma(w)\dotsm\sigma^{i-1}(w))=il(w)$ for all $i>0$. They were first studied systematically by He and Nie \cite{HN}. Each $\sigma$-conjugacy class contains a representative of a straight element. However, in general there are still finitely many (not necessarily $K$-$\sigma$-conjugate) such representatives. It would be interesting to see how the induced stratifications vary if one chooses different straight representatives of the same $\sigma$-conjugacy class, or if there are examples where it is more useful to consider a totally different representative $b\in[b]$.
\end{remark}

Before we proceed let us give an example.

\begin{ex}\label{rmk_description_invariant_GLn}For $G=GL_n$, let $N=L^n$ and $\Lambda_0$ the $\mathcal{O}_L$-lattice generated by the standard basis. Then $K$ is the stabilizer of $\Lambda_0$ and we obtain a bijection between $G(L)/K$ and the set of lattices in $N$ by mapping $g$ to $g\Lambda_0$. If $g_1,g_2\in G(L)/K$, then $f_{g_1}=f_{g_2}$ if and only if
\begin{equation}\label{thisformula} \vol(g_1\Lambda_0\cap j\Lambda_0)=\vol(g_2\Lambda_0\cap j\Lambda_0) \text{ for all } j\in J_b(F).\end{equation} Here, $\vol$ denotes the volume of a lattice defined as $$\vol(\Lambda)=\rg(\Lambda_0/(\Lambda\cap\Lambda_0))-
\rg(\Lambda/(\Lambda\cap\Lambda_0)).$$ Indeed, \eqref{thisformula} follows from the fact that
$\vol(g\Lambda_0\cap j\Lambda_0)=\sum_{\mu_i>0}\mu_i+v_p(\mathrm{det} j)$ where $f_g(j)=(\mu_i)\in X_*(T)_{\dom}\subset \Z^n$, applied to all elements $p^ij$ for $i\in\Z$.
\end{ex}

We come back to the general case.

\begin{lemma}\label{lemggad}
Let $G^{\ad}$ be the adjoint group of $G$, and let $g,h\in G(L)$ with images $\bar g,\bar h\in G^{\ad}(L)$. Let $f_{\bar g}$ and $f_{\bar h}$ denote the analogous functions associated to $\bar g,\bar h$ and to the group $J_{\bar b}(F)$ associated with $(G^{\ad},\bar b)$ where $\bar b$ is the image of $b$. If $f_{\bar g}=f_{\bar h}$ and $\kappa_G(g)=\kappa_G(h)$ in $\pi_1(G)$ then $f_g=f_h$. The converse holds if the projection $J_b(F)\rightarrow J_{\bar b}(F)$ is surjective.
\end{lemma}
\begin{proof}
 The value $\kappa_G(g)$ is determined by $KgK$, and hence by $f_g(1)$. For each $j\in J_b(F)$ the class $Kj^{-1}gK$ is mapped to the class of $\bar j^{-1}\bar g$ in $G^{\ad}$. This implies the last assertion for groups where $J_b(F)\rightarrow J_{\bar b}(F)$ is surjective.

For the other direction, the above argument implies that $Kj^{-1}gK$ and $Kj^{-1}hK$ coincide up to a scalar $\lambda_j$ in the center $Z$ of $G$.  The restriction of the projection map $Z(L)/(Z(L)\cap K)\rightarrow \pi_1(G)=X_*(T)/(\text{ coroot lattice })$ is injective and coincides with the map induced by $\kappa_G$. Thus $\lambda_j$ is trivial if and only if $\kappa_G(j^{-1}g)=\kappa_G (j^{-1}h)$. As $\kappa_G$ is a group homomorphism, this is equivalent to $\kappa_G(g)=\kappa_G(h)$.
\end{proof}

\begin{remark}\label{remggad}
For a datum $(G,b,\mu)$ where $J_b(F)\rightarrow J_{\bar b}(F)$ is not  surjective (e.g. $G=\mathrm{SL}_n$ with $b$ superbasic), it may be useful to consider the finer invariant on $X_{\mu}^G(b)$ induced by the functions $f_{\bar g}$ on $X_{\mu}^{G^{\ad}}(b)$, together with the value $\kappa_G(g)$. One easily sees that this invariant has analogous properties (for example concerning finiteness as studied below). In the next section we will also use this finer invariant.
\end{remark}

Our next aim is to provide finiteness results on the set of possible functions $f_g$ and on the induced decomposition of the affine Grassmannian. They are based on the following proposition.
\begin{prop}\label{propfin1}
Let $\mu_0\in X_*(T)_{\dom}$. Then the set of functions $f_g$ (for $g\in G(L)$) with $f_g(1)\preceq \mu_0$ (i.e. with $g\in \overline{K\mu_0K}$) is finite.
\end{prop}
Here, $\preceq$ denotes the partial ordering on $X_*(T)_{\dom}$ such that $\mu'\preceq\mu$ if and only if $\mu-\mu'$ is a positive integral linear combination of positive coroots.

Notice that the role of $\mu_0$ in the proposition is a completely different one as the role of $\mu$ in the definition of $X_{\mu}(b)$. Here we bound $g$, whereas before, one bounds $g^{-1}b\sigma(g)$.
\begin{proof}
We first remark that for any fixed $j$, the values $f_{g}(j)$ for $g$ as above lie in a finite subset of $X_*(T)_{\dom}$. Indeed, we have $g\in K\mu'K$ for some $\mu'\preceq\mu_0$. Thus if $j^{-1}\in K\mu_1K$ for some dominant cocharacter $\mu_1$, then $j^{-1}g\in K\mu_1K\mu'K\subseteq \overline{K(\mu_1+\mu')K}$ which implies our remark.

We may replace $b$ by a representative of $[b]$ that is defined over a finite unramified extension of $F$. Indeed, we replace $b$ by any element $b'=bk_0$ with $k_0\in K$ and $k_0\equiv 1\pmod {\e^m}$. If $m$ is sufficiently large, $b'$ is $\sigma$-conjugate to $b$ via an element $k_1\equiv 1\pmod {\e^{m'}}$ for some $m'$. Also, $m'$ can be assumed to be large by taking $m$ sufficiently large. Then conjugation by $k_1$ induces the isomorphism $J_b(F)\rightarrow J_{b'}(F)$ and multiplication by $k_1^{-1}$ the isomorphism $X_{\mu}(b)\rightarrow X_{\mu}(b')$. In particular for $j\in J_b(F)$ and all $g$ we have $Kj^{-1}gK=Kk_1(k_1^{-1}jk_1)(k_1^{-1}g)K$, that is $f_g^b(j)=f_{k_1^{-1}g}^{b'}(k_1^{-1}jk_1)$. Here we denote by $f_g^b$ resp. $f_g^{b'}$ the functions defined on $J_b(F)$ resp. on $J_{b'}(F)$. Furthermore, if $m'$ is large (depending on $\mu_0$), $k_1^{-1}gK=gK$ for all $g\in \overline{K\mu_0K}$. Therefore the proposition holds for the functions $f_g^b$ if and only if it holds for the functions $f_g^{b'}$. We can now choose $k_0$ in such a way that $b'=bk_0$ is defined over a finite unramified extension of $F$. We replace $b$ by $b'$ and from now on assume that the same holds for $b$.

We want to show that there are finitely many elements $j\in J_b(F)$ such that the $f_g(j)$ for these elements determine all other values of $f_g$, for each $g$ as above. This together with the remark we just proved then implies the proposition.

By \cite{HV1}, Lemma 3.11, we have the following alternative way to compute $f_g(j)$. Let $B$ be the chosen Borel subgroup and let $\overline{B}$ be its opposite. As $G$ is unramified over $F$, it is split over some finite unramified extension $F'$ of $F$. Therefore for any dominant weight $\lambda$ of $G$, we can define the Weyl module $V(\lambda)$ of $G_{/F'}$ with highest weight $\lambda$. Indeed, $V(\lambda)=\bigl(\Ind_{\overline{B}}^G(-w_0\lambda)\bigr)^{\vee}$ with $w_0$  the longest element in the (finite) Weyl group of $G$. It is an $F'$-representation of $G$ generated by a $B$-stable line on which $B$ acts through $\lambda$. We denote the associated representation by $\rho_{\lambda}$. Since $G$, $B$, $\overline{B}$ and $\lambda$ can all be defined over $\O_L$, the $L$-representation $(V(\lambda)\otimes L,\rho_{\lambda})$ comes in fact from a $\O_L$-representation $(M_{\lambda},\rho_{\lambda})$ with $V(\lambda)\otimes_{F'}L=M_{\lambda}\otimes_{\O_L} L$. Then $f_g(j)\preceq\mu_1$ if and only if
$$\rho_{\lambda}(j^{-1}g)(M_{\lambda})\subseteq \e^{\langle w_0\lambda,\mu_1\rangle}\cdot M_{\lambda}\text{ for all }\lambda\text{ and}$$
 $$\mu_1=\kappa_G(j^{-1}g)=\kappa_G(g)-\kappa_G(j)=\mu_0-\kappa_G(j)\text{ as elements of }\pi_1(G)$$
where $\kappa_G$ denotes the Kottwitz homomorphism.  The precise value of $f_g(j)$ can then be computed as the minimum (wrt $\preceq$) of all $\mu_1$ such that the above conditions hold. By \cite{HV1}, Lemma 3.7, the first condition holds if and only if it holds for a finite generating set of the monoid of dominant weights. The second condition only depends on $\mu_0$ and not on $g$, and can therefore be neglected for the rest of the proof.

For a dominant weight $\lambda$ let $m_{\lambda,g}:J_b(F)\rightarrow \mathbb{Z}$ where $m_{\lambda,g}(j)$ is the maximum of all integers $m$ such that
\begin{equation}\label{eqdefm}
\rho_{\lambda}(j^{-1}g)(M_{\lambda})\subseteq \e^{m}\cdot M_{\lambda}.
\end{equation}
By what we just saw, there is a finite set of weights $\lambda$ such that the functions $m_{\lambda,g}$ together determine $f_g$. Also, for fixed $j$ and $\lambda$, and for varying $g$ as above, there are only finitely many values of $m_{\lambda,g}(j)$.

Let $m_{\lambda}, m'_{\lambda}\in \mathbb{Z}$ with
\begin{equation}
\label{eqm1}\e^{m'_{\lambda}}( M_{\lambda})\subseteq\rho_{\lambda}(g)(M_{\lambda})\subseteq  \e^{m_{\lambda}}( M_{\lambda})
\end{equation}
for all $g$ as above. (For example, we can choose $m_{\lambda}=\langle \mu_0, w_0\lambda\rangle$ and $m'_{\lambda}=\langle\mu_0, \lambda\rangle$ since we have $w_0\lambda\preceq\lambda'\preceq\lambda$ for any weight $\lambda'$ of $V(\lambda)$,which implies
\[\langle \mu_0, w_0\lambda\rangle\leq \langle f_g(1), w_0\lambda\rangle\leq \langle f_g(1), \lambda'\rangle\leq\langle f_g(1),\lambda\rangle\leq \langle \mu_0, \lambda\rangle.)\] Then it is enough to show

{\it Claim 1.} There is a finite subset $I=I_{\lambda}$ in $J_b(F)$ with the following property. Assume that $g,g'\in G(L)$ satisfy \eqref{eqm1} and $m_{\lambda,g}(j)=m_{\lambda,g'}(j)$ for all $j\in I$. Then $m_{\lambda,g}=m_{\lambda,g'}$.

To prove this claim we first show

{\it Claim 2.} There are only finitely many lattices of the form $$\e^{m_{\lambda}}(M_{\lambda})\cap(\e^m\rho_{\lambda}(j)(M_{\lambda})+\e^{m'_{\lambda}}(M_{\lambda}))$$ with $j\in J_b(F)$.

We show claim 2. All of these lattices contain the lattice $\e^{m'_{\lambda}}(M_{\lambda})$ and are contained in $\e^{m_{\lambda}}(M_{\lambda}).$ Furthermore, after replacing $F'$ by some finite unramified extension, we may assume that each $\rho_{\lambda}(j)$ is defined over $F'$ (as $j\in J_b(F)$, and $b$ is assumed to be defined over some finite unramified extension of $F$),  and thus the same holds for the lattices we want to consider. In other words, each such lattice $\Lambda$ satisfies $\Lambda=(\Lambda\cap V(\lambda))\otimes_{F'}L$. But $(V(\lambda)\cap \e^{m_{\lambda}}(M_{\lambda}))/\e^{m'_{\lambda}}(M_{\lambda})$ is a finite set, which implies claim 2.

To prove claim 1 we remark that \eqref{eqdefm} is equivalent to
$\rho_{\lambda}(g)(M_{\lambda})\subseteq \e^{m}\cdot\rho_{\lambda}(j) M_{\lambda}$ or (by \eqref{eqm1}) to the two conditions
\begin{eqnarray}
\nonumber \e^{m'_{\lambda}}( M_{\lambda})&\subseteq&\e^{m}\cdot\rho_{\lambda}(j) M_{\lambda}\\
\label{eqfin1}\rho_{\lambda}(g)(M_{\lambda})&\subseteq& (\e^{m}\cdot\rho_{\lambda}(j) M_{\lambda}+ \e^{m'_{\lambda}}( M_{\lambda}))\cap \e^{m_{\lambda}}( M_{\lambda}).
\end{eqnarray}
The first of these conditions is independent of $g$. Choose $j_1,\dotsc, j_l\in J_b(F)$ and $m_1,\dotsc,m_l$ such that the $(\e^{m_i}\cdot\rho_{\lambda}(j_i) M_{\lambda}+ \e^{m'_{\lambda}}( M_{\lambda}))\cap \e^{m_{\lambda}}( M_{\lambda})$ are the finitely many lattices of claim 2. Let $I=\{j_1,\dotsc, j_l\}$ and assume that $m_{\lambda, g}(j)=m_{\lambda,g'}(j)$ for all $j\in I$. Consider now an arbitrary $j\in J_b(F),$ and $m\in \mathbb Z$. Let $j_i\in I$ with $$(\e^{m}\rho_{\lambda}(j) M_{\lambda}+ \e^{m'_{\lambda}}( M_{\lambda}))\cap \e^{m_{\lambda}}( M_{\lambda})=(\e^{m_i}\rho_{\lambda}(j_i) M_{\lambda}+ \e^{m'_{\lambda}}( M_{\lambda}))\cap \e^{m_{\lambda}}( M_{\lambda}).$$ Then \eqref{eqfin1} holds for $j,m,$ and $g$ (resp.~$g'$) if and only if it holds for $j_i,m_i,$ and $g$ (resp. $g'$). As $m_{\lambda, g}(j_i)=m_{\lambda,g'}(j_i)$, these conditions for $g$ and $g'$ are equivalent.
\end{proof}
We now apply this to affine Deligne-Lusztig varieties.
\begin{cor}
Let $\mu\in X_*(T)_{\dom}$ and let $b\in B(G,\mu)$. Then the action of $J_b(F)$ on $\{f_g\mid g\in X_{\mu}(b)\}$ has finitely many orbits.
\end{cor}
\begin{proof}
Indeed, by \cite{RZIndag}, for each $g\in X_{\mu}(b)$ there is a $j\in J_b(F)$ such that $jg\in \overline {K\mu_0K}$ for some $\mu_0$ only depending on $\mu$ and $b$. Thus every $J_b(F)$-orbit contains a function $f_g$ with $f_g(1)\preceq\mu_0$, and the assertion follows from the proposition above.
\end{proof}

\begin{prop}\label{propfin}
Let $g\in G(L)$ and let $f=f_g$ be the associated function.
\begin{enumerate}
\item  There are finitely many elements $j_1,\dotsc,j_d\in J_b(F)$ such that for all $h\in G(L)$ with $f_h(j_i)=f(j_i)$ for all $i$ we have $f=f_h$.
\item  There are finitely many elements $j_1,\dotsc,j_{d'}\in J_b(F)$ such that for all $h\in G(L)$ with $f_h(j_i)\preceq f_g(j_i)$ for all $i$ we have $f_h(j)\preceq f_g(j)$ for all $j\in J_b(F)$.
\item There are finitely many elements $j_1,\dotsc,j_{d''}\in J_b(F)$ such that for all $h,h'\in G(L)$ with $f_h(j_i)=f_{h'}(j_i)\preceq f_g(j_i)$ for all $i$ we have $f_h(j)=f_{h'}(j)\preceq f_g(j)$ for all $j\in J_b(F)$.
\item The set $S_{f}=\{h\in G(L)/K\mid f_h=f\}$ defines a locally closed subspace of the affine Grassmannian.
\item The set $S_{\preceq f}=\{h\in G(L)/K\mid f_h(j)\preceq f(j)\text{ for all }j\in J_b(F)\}$ is closed and a finite union of sets $S_{f'}$.
\end{enumerate}
\end{prop}
\begin{proof}
Clearly, the third assertion implies the first and the second. Let us now prove the third assertion. Let $j_1=1$ and $\mu_0=f_g(1)$. Then by Proposition \ref{propfin1} there are only finitely many functions $f_h$ with $f_h(1)\preceq f_g(1)$. Thus we can choose $j_2,\dotsc,j_{d''}\in J_b(F)$ such that the values at the $j_i$ distinguish these finitely many functions.

To prove the last two assertions we use the well-known results on Schubert cells in affine Grassmannians (which also hold in the arithmetic case by \cite{ZhuGr}, 1.5). We obtain that for some fixed $j\in J_b(F)$ and $\mu\in X_*(T)_{\dom}$, the set of all $g\in G(L)$ with $f_g(j)\preceq \mu$ is closed, and the subset where $f_g(j)= \mu$ is open and dense in it. Thus the first assertion implies the fourth, and the third assertion implies the fifth.
\end{proof}

\subsection{Closures of strata}\label{secclos}

We give an explicit example which is interesting in itself, but also showing that the decomposition of $X_{\mu}(b)$ is in general not a stratification in the proper sense, i.e.~that the closure of some $S_f$ is in general not a union of other strata.

For this example let $G=GL_{2m}$ for some $m$, let $B$ be the Borel subgroup of upper triangular matrices and let $T$ be the diagonal torus. Let $N=L^{2m}$ with a basis $e_{ij}$ ($1\leq i\leq m$, $j=1,2$) and we define $e_{ij}$ for $j\in\Z$ by the rule $e_{i,j+2}=\e e_{ij}$. Let $b\in G(L)$ with $b(e_{ij})=e_{i,j+1}$. Then $b$ is basic of slope $1/2$. Furthermore, $J_b(F)=GL_m(D)$ where $D$ is the quaternion division algebra over $F$. In particular, we have $b\in J_b(F)$.

Let $\Lambda_0=\O_L^{2m}$ be the standard lattice, generated by our basis $e_{ij}$ ($1\leq i\leq m$, $j=1,2$). Then elements of $G(L)/K$   are in bijection with lattices $\Lambda\subset N$, by mapping $g$ to $g\Lambda_0$. Let $\mu\in X_*(T)_{\dom}$ be minuscule with $[b]\in B(G,\mu)$, i.e.~$\mu=(1^{(m)},0^{(m)})$. Then under the above bijection points of $X_{\mu}(b)$ correspond to lattices satisfying $\Lambda\supseteq b\sigma(\Lambda)\supseteq \e\Lambda$.

 We consider the subscheme of the affine Grassmannian consisting of all lattices $\Lambda$ with $$b\Lambda_0\subseteq\Lambda\subseteq \Lambda_0,$$ and such that $\Lambda$ has some fixed volume $m-d$. One easily verifies that this is a (closed) subscheme of $X_{\mu}(b)$. Furthermore, as $1,b\in J_b(F)$, it is a union of (finitely many) strata $S_f$. A lattice $\Lambda$ in this subset is determined by its image $\bar\Lambda\subseteq\Lambda_0/b\Lambda_0\cong k^m$ which is a sub-vector space of dimension $d$. Therefore this closed subscheme of the affine Grassmannian is isomorphic to the (classical) Grassmannian $Gr_d({\overline k}^m)$. A lattice $\Lambda$ is of the form $j\Lambda_0$ for some $j\in J_b(F)$ if and only if $\bar\Lambda$ is defined over $\mathbb{F}_{q^2}$. 
 
{\it Claim.} Fixing the function $f_g$ (with $\Lambda=g\Lambda_0$) corresponds to fixing the dimensions of all intersections of $\bar\Lambda$ with sub-vector spaces defined over $\mathbb{F}_{q^2}$. 

We proceed as in the proof of Proposition \ref{propfin1}: Fixing the relative position of $g$ and some fixed $j\in J_b(F)$ is equivalent to fixing the dimensions of all intersections $\Lambda\cap \epsilon^i j\Lambda_0$ for all $i\in \mathbb{Z}$. Furthermore as $b\Lambda_0\subseteq \Lambda\subseteq \Lambda_0$, such a dimension of an intersection is determined by the dimension of the intersection of $\bar\Lambda$ with $((\epsilon^i j\Lambda_0+b\Lambda_0)\cap \Lambda_0)/b\Lambda_0$. This is a sub-vector space of $\Lambda_0/b\Lambda_0$ defined over $\mathbb{F}_{q^2}$. The claim follows. 
 
The induced decomposition of $Gr_d({\overline k}^m)$ into locally closed subschemes refines the decomposition into thin Schubert cells of Gelfand, Goresky, MacPherson, and Serganova \cite{GGMS}. There, one fixes the dimensions of the intersections with all sub-vector spaces generated by some of the coordinate vectors.

 For $d\leq m$, every $d\times m$ matrix $A$ of rank $d$ with coefficients in $\bar k$ determines an element $W_A\in Gr_d(\bar k^m)$ such that $W_A$ is generated by the $d$ rows of $A$. Let $I_d$ be the set of subsets of $d$ elements in $\{1,2,\dotsc, m\}$. For any $J\in I_d$, let $A_J$ be the $d\times d$ matrix whose columns are the columns of $A$ at indices from $J$. Let $L_A=\{J\in I_d| \det A_J\neq 0\}$. In fact, $L_A$ only depends on $W_A$, and we can also associate with every element $W$ of $Gr_d(\bar k^m)$ a corresponding set $L_{W}$.  It is called the list of $W$. The locally closed subschemes of $ Gr_d(\bar k^m)$ given by fixing $L_W$ are called thin Schubert cells, cf. \cite{GS}. For $W, W'\in Gr_d(\bar k^m)$, if $W'$ is in the closure of the thin Schubert cell of $W$, then $L_{W'}\subseteq L_W$.

The restriction of the action of $GL_m(\bar k)$ induces an action of $GL_m(k)$ on $Gr_d(\bar k^m)$. Our decomposition is the intersection of the decomposition of thin Schubert cells under the action of $GL_m(k)$. For any $W\in Gr_d(\bar k^m)$, our invariant for $W$ is determined by $(L_{\varphi W})_{\varphi\in GL_m(k)}$. And for $W, W'\in Gr_d(\bar k^m)$, if $W'$ is in the closure of the stratum of $W$, then $\varphi(W')$ is in the closure of the thin Schubert cell of $\varphi(W)$ for all $\varphi\in GL_m(k)$.

In \cite{GS} 1.11.1, Gelfand and Serganova give an example for $m=7$ showing that in general, the closure of thin Schubert cells in the Grassmannian $Gr_3( \mathbb{C}^7)$ of $3$-dimensional sub-vector spaces of $\mathbb{C}^7$ is not a union of thin Schubert cells. In the rest of this section, we will give an analogous example in our case for $m=7$, which shows that our decomposition is not a stratification. So from now on assume $m=7$ and $d=3$.

For $\lambda\in\bar k$, let $$A_{\lambda}=\bigg (\begin{array}{ccccccc} 1 &0  &0 &x_1\lambda &0 &x_1x_5 &x_1x_6\\ 0 &1 &0 &x_2\lambda &x_2x_4 &0 &x_2x_6\\ 0 &0 &1 &x_3\lambda &x_3x_4 &x_3x_5 &0\end{array}\bigg )$$
with $x_1,\dotsc, x_6\in\bar k$ such that we have for the degrees of the fields extensions \begin{eqnarray}\label{eqn_generic_x_i}[k(x_1,\dotsc, x_i): k(x_1,\dotsc, x_{i-1})]>>0 \text{ for } i=1,\dotsc, 6.\end{eqnarray} When $\lambda\neq 0$,  $L_{A_\lambda}\subset I_3$ is a subset of 29 elements. Indeed, if we consider the $i$-th column of $A_{\lambda}$ as the homogeneous coordinates of a point $v_i \in\mathbb{P}^2(\bar k)$, then $J=\{i_1, i_2, i_3\}\in L_A$ if and only if $v_{i_1}, v_{i_2}, v_{i_3}$ are not on a line in $\mathbb{P}^2(\bar k)$. The positions of the $(v_i)_{1\leq i\leq 7}$ in $\mathbb{P}^2(\bar k)$ are as follows.

{\setlength{\unitlength}{10 pt}
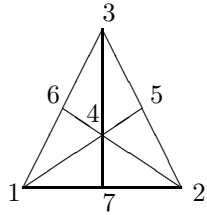
\begin{figure}[h]
\begin{center}
\begin{picture}(9,7)(1,0)

\put(1,1){\line(1,0){6}}
\put(4,1){\line(0,1){6}}
\put(1,1){\line(1,2){3}}
\put(4,7){\line(1,-2){3}}
\put(1,1){\line(3,2){4.5}}
\put(7,1){\line(-3,2){4.5}}

\put(0.4,0.5){$1$}
\put(4,0.1){$7$}
\put(7.4,0.5){$2$}
\put(3.4,3.5){$4$}
\put(1.9,4.2){$6$}
\put(5.8,4.2){$5$}
\put(4,7.3){$3$}
\end{picture}
\end{center}
\end{figure}

Let $W_{\lambda}=W_{A_{\lambda}}\in Gr_3(\bar k^7)$. We will show that the closure of the stratum of $W_1$ is not a union of strata. For any $\varphi\in GL_7(k)$ and $J\in I_3$, by the Cauchy-Binet formula
\[\det (A_{\lambda}\cdot\varphi)_J=\sum_{J'\in I_3} \det (A_\lambda)_{J'}\det \varphi_{J', J},\]
where $\varphi_{J', J}$ is the $3\times 3$ matrix whose rows (resp. colomns) are the rows (resp. colomns) of $\varphi$ at indices from $J'$ (resp. $J$). Notice that $\det \varphi_{J', J}\in k$ and when $\lambda=1$, the 29 non-zero minors of $A_1$ are linearly independent over $k$ by hypothesis (\ref{eqn_generic_x_i}). Therefore for all but finitely many $\lambda\in\bar k$, $W_{\lambda}$ has the same invariant $(L_{\varphi W})_{\varphi\in GL_7(k)}$ as $W_1$ and hence $W_\lambda$ is in the closure of the stratum of $W_1$ for all $\lambda$. In particular $W_0$ is in the closure of the stratum of $W_1$.

Choose $\gamma\in \bar k$ such that the degree $[k(x_1,\dotsc, x_6, \gamma): k(x_1,\dotsc, x_6)]$ is sufficiently large. Let
\[B=\bigg (\begin{array}{ccccccc} 1 &0  &0 &0 &0 &x_1x_5 &x_1x_6\\ 0 &1 &0 &0 &x_2x_4 &0 &\gamma\\ 0 &0 &1 &0 &x_3x_4 &x_3x_5 &0\end{array}\bigg )\] We want to show that $W_B$ is in the same stratum as $W_0$, but $W_B$ is not in the closure of the stratum of $W_1$. Indeed, it is easy to check that $L_{A_0}=L_B$ both have 17 elements and the 17 non-zero minors of $A_0$ (resp. $B$) are linearly independent over $k$. This implies that $L_{\varphi W_0}=L_{\varphi W_B}$ for all $\varphi\in GL_7(k)$ and hence $W_{0}$ and $W_B$ are in the same stratum. On the other hand, as $W_B$ is not in the closure of the thin Schubert cell of $W_1$ (\cite{GS},1.11.1), it is not in the closure of the stratum of $W_1$.

\section{Relation to semi-modules}\label{secsemmod}

In this section we consider the case that $[b]$ is superbasic. By definition, this means that no representative of $[b]$ is contained in a proper Levi subgroup of $G$ defined over $F$. By \cite{ckv}, Lemma 3.1.1 this implies that $G^{\ad}$ is a product of groups of the form $\Res_{F_i|F}PGL_{h_i}$ where the $F_i$ are finite unramified extensions of $F$. Our aim is to compare our stratification to the stratification by semi-modules introduced by de Jong and Oort \cite{deJongOort} for Rapoport-Zink spaces for the (split) $GL_h$-case, resp.~to its natural generalization, the (weak) EL-charts introduced by Hamacher in \cite{hamacher} for affine Deligne-Lusztig varieties and Rapoport-Zink spaces for the superbasic case for unramified groups. These stratifications are first defined for affine Grassmannians and affine Deligne-Lusztig varieties for groups of the form $\Res_{F'|F}GL_{h}$ for some unramified extension $F'$ of $F$ and some $h$.  As the affine Grassmannian resp.~an affine Deligne-Lusztig variety for a product of groups is a product of the individual varieties this induces invariants for products of such groups. And also note that for $\lambda\in\pi_1(G)$,
\begin{eqnarray*}\mathrm{Gr}_{G}^{\lambda}&\simeq& \mathrm{Gr}_{G^{ad}}^{\lambda_{ad}},\\ X^G_{\mu}(b)^{\lambda}&\simeq& X^{G^{ad}}_{\mu_{ad}}(b_{ad})^{\lambda_{ad}},\end{eqnarray*}
 where $\mathrm{Gr}_{G}^{\lambda}$ is the connected component of the (twisted) affine Grassmannian $\mathrm{Gr}_G$ of $G$ which is the fiber of $\lambda$ for the Kottwitz map $\mathrm{Gr}_G\rightarrow \pi_1(G)$, $X^G_{\mu}(b)^{\lambda}=X^G_{\mu}(b)\cap\mathrm{Gr}_G^{\lambda}$, and $\lambda_{ad}, b_{ad}, \mu_{ad}$ are induced from $\lambda, b, \mu$ respectively via the natural projection $G\rightarrow G^{ad}$.   Hence the decomposition of the affine Grassmannian and of affine Deligne-Lusztig varieties by EL-charts for  groups of the form $\Res_{F'|F}GL_{h}$ induces the decomposition for general superbasic case. Our main result, Proposition \ref{propsuperb} shows that this decomposition induced by weak EL-charts coincides with the decomposition of $X_{\mu}^G(b)$ by the functions $f_{\bar g}:J_{\bar{b}}(F)\rightarrow X_*(T')_{\dom}$ where $T'$ is the image of $T$ in $G^{\ad}$, and $\bar g\in G^{\ad}(L)$, compare Remark \ref{remggad}.

By definition we consider the functions $f_{\bar g}$, i.e.~we may assume that $G$ is adjoint. As the affine Grassmannian resp.~an affine Deligne-Lusztig variety for a product of groups is a product of the individual varieties, we may assume that $G=\Res_{F'|F}PGL_{h}$ for some $F'$ and $h$. Using Lemma \ref{lemggad} and the surjectivity of $J_b(F)\rightarrow J_{\bar b}(F)$ for the group $\Res_{F'|F}GL_{h}$ we may assume that already $G=\Res_{F'|F}GL_{h}$. Let $d$ be the degree of the extension $F'|F$, and denote by $\mathcal{I}$ the set of embeddings $\tau: F'\rightarrow \overline F$ into a fixed algebraic closure of $F$. Then the $d$ elements of $\mathcal{I}$ are permuted cyclically by Frobenius.

Let us recall some facts on superbasic conjugacy classes, a reference for this being Section 3.1 of \cite{ckv} and \cite{hamacher}. For $G$ as above, each superbasic $\sigma$-conjugacy class has a representative of the following explicit kind. A superbasic $\sigma$-conjugacy class has Newton slope $\frac{m}{dh}$ for some $m$ coprime to $h$. Fix for each $\tau\in \mathcal{I}$ some $m_{\tau}\in\mathbb{Z}$ such that $\sum_{\tau} m_{\tau}=m$. Then consider $N=\oplus_{\tau\in \mathcal{I}}L^h$ and denote the standard basis of the $\tau$-th of these vector spaces by $e_{1,\tau},\dotsc,e_{h,\tau}$. For $i\in \mathbb{Z}$ we define $e_{i,\tau}$ by $e_{i+h,\tau}=\e e_{i,\tau}$. Then we define $b\in G(L)$ by $b(e_{i,\tau})=e_{i+m_{\tau},\tau}$, hence $b\sigma(e_{i,\tau})=e_{i+m_{\sigma\tau},\sigma\tau}$. One can now check explicitly that this element has the given Newton slope and is hence a representative of the given class. Note that all of these representatives are of length 0 and hence straight.

We give a weakened version of Hamacher's definition of an EL chart, and of the EL chart associated with a lattice (compare  \cite{hamacher}). A weak EL chart for the group $G$ is a subset $A\subset \coprod_{\tau} \mathbb{Z}$ which is bounded from below and contains all $(i,\tau)$ for $i$ sufficiently large, and which satisfies $A+h\subseteq A$. Let $0\neq v\in L^h$, considered as the $\tau$th component of $N$. Then we can write $v$ uniquely as an infinite converging sum $v=\sum_{i\geq i_0}[a_i]e_{i,\tau}$ where the $[a_i]$ are (Teichm\"uller or usual, depending on the characteristic of $L$) representatives of $a_i\in \overline k$ in $L$. The pair $(i,\tau)$ such that $i$ is minimal with $a_i\neq 0$ is called the start index of $v$. Let $\Lambda_0$ be a lattice in $N$ with a decomposition $\Lambda_0=\oplus \Lambda_{0,\tau}$. Then the set of start indices of all $v\in \Lambda_{0,\tau}\setminus \{0\}$ (for all $\tau$) is called the weak EL chart of $\Lambda_0$. If $v$ has start index $(i,\tau)$, then $\e v$ has start index $(i+h,\tau)$. This implies that the weak EL chart of $\Lambda_0$ is indeed a weak EL chart. Note that the EL charts in the sense of Hamacher are also weak EL charts, and that the lattices associated with an affine Deligne-Lusztig variety with only non-negative Hodge slopes have weak EL charts which are automatically EL charts in the stronger sense.

\begin{remark}\label{group_semi_module}There is also a group-theoretic way to describe the decomposition of affine Grassmannian and affine Deligne-Lusztig varieties by weak EL-charts. For any quasi-split and unramified group $G$ over $F$, we have the IAK decomposition:
\[G(L)=\coprod_{\mu\in X_*(T)} I\epsilon^{\mu}K\] with $I\subset K$ an Iwahori subgroup of $G(L)$. From the point of view of the Bruhat-Tits building of $G$ over $L$, this decomposition reflects the fact that for any special point and any chamber in the building, there exists an apartment containing both of them. For $G=\Res_{F'|F}GL_{h}$ we use $G(L)=\prod_{\tau\in\mathcal{I}} GL_h(L)$. We take the Iwahori subgroup $I$ to be the preimage of the group consisting of lower triangular matrices for each component $\tau$ via the map $G(\mathcal{O}_L)\rightarrow G(\bar k)$ which maps $\epsilon$ to 0. We can easily check that the decomposition of the affine Grassmannian and of the affine Deligne-Lusztig varieties by weak EL-charts in this case coincide with the IAK decomposition. More precisely, take $T=\Res_{F'|F}\mathbb{G}_m^h$, then there is a natural identification $X_*(T)=\prod_{\tau\in\mathcal{I}}\mathbb{Z}^h$. For any $\mu=(\mu_{1,\tau}, \cdots, \mu_{h, \tau})_{\tau\in\mathcal{I}}\in X_*(T)$, then $\mu$ corresponds to the EL-chart $A^{\mu}=\coprod_{\tau\in\mathcal{I}} A^{\mu}_{\tau}$ with $A^{\mu}_{\tau}=\{i+\mu_{i, \tau}h+\mathbb{N}h| 1\leq i\leq h\}$.
For general $G$  we have the same assertion if we take the Iwahori subgroup $I\subset G(L)$ such that its image under $G(L)\rightarrow G^{ad}(L)=\prod_{\tau\in\mathcal{I}}\mathrm{PGL}_{h}(L)$ corresponds to the Borel subgroup of the lower triangular matrices for each component.
\end{remark}

\begin{ex}\label{exelchart}
In general, there are many lattices having the same weak EL chart. An exception (in fact the only one) is the following example.

For each $\tau\in \mathcal{I}$ we fix an element $l_{\tau}\in\mathbb{Z}$ and let $A_{\tau}=\mathbb{Z}_{\geq l_{\tau}}$. Let $\Lambda=\bigoplus_{\tau} \Lambda_{\tau}$ be a lattice with weak EL chart $A=\coprod_{\tau}A_{\tau}$. We claim that $\Lambda$ is the lattice generated by all $e_{i,\tau}$ with $i\geq l_{\tau}$. Indeed, let $v\in \Lambda_{\tau}$. Then the start index of $v$ is greater or equal to $l_{\tau}$, and hence it is contained in the lattice generated by all $e_{i,\tau}$ with $i\geq l_{\tau}$. On the other hand let $i_0$ be maximal with $e_{i_0,\tau}\notin \Lambda_{\tau}$. We have to show that $i_0<l_{\tau}$. Assume that $i_0\geq l_{\tau}$. Let $w\in \Lambda$ with start index $(i_0,\tau)$ (which is possible as this is contained in the EL chart $M$ of $\Lambda$). Then $e_{i_0,\tau}=w+\sum_{i>i_0}[c_i]e_{i,\tau}$ for some $c_i$.  By maximality of $i_0$, $e_{i,\tau}\in \Lambda_{\tau}$ for all $i>i_0$. Hence $e_{i_0,\tau}\in\Lambda_{\tau}$ which contradicts $e_{i_0,\tau}\notin\Lambda_{\tau}$ and shows that we must have $i_0<l_{\tau}$.
\end{ex}

In the superbasic case (and for $G$ as above) we have the following explicit description of $J_b(F)$ and of the quotient that we need to consider for the function $f_g$.
\begin{lemma}\label{lemjsuperbasic}
Let $G=\Res_{F'|F}\GL_h$ with $F'/F$ unramified of degree $d$ and let $b\in G(L)$ be superbasic of slope $\frac{m}{dh}$. Then $J_b(F)$ is the central division algebra over $F'$ of rank $h^2$ and invariant $\frac{m}{h}$. In particular, $\mathbb{Z}\cong J_b(F)/(J_b(F)\cap K)$ where an isomorphism is given by $l\mapsto \pi^l\cdot(J_b(F)\cap K)$ with $\pi\in J_b(F)$ a uniformizer which sends $e_{\tau,i}$ to $e_{\tau,i+1}$ for all $i\in\mathbb{Z}$ and $\tau\in \mathcal{I}$ if we regard it as an element in $G(L)$.
\end{lemma}
\begin{proof}
The corresponding statement for the split case $G=\GL_h$ is known, compare \cite{deJongOort}, 5.4. We show how to reduce the general statement to this. We have $G(L)\cong\prod_{\tau\in \mathcal{I}} \GL_h$. Let $j=(j_{\tau})$ be the decomposition of an element of $J_b(F)\subseteq G(L)$.  We fix one embedding $\tau_0\in I.$ By definition of $J_b(F)$, we have $j_{\sigma\tau}\circ b_{\sigma\tau}=b_{\sigma\tau}\circ \sigma(j_{\tau})$ for all $\tau$. Hence $j_{\tau_0}\in \GL_h(L)$ determines $j$, and has to satisfy the condition $j_{\tau_0}\circ (b\sigma)^d_{\tau_0}=(b\sigma)^d_{\tau_0}\circ j_{\tau_0}$. But $(b\sigma)^d$ is superbasic of slope $\frac{m}{d}$, and the first assertion follows. The second is an immediate consequence.
\end{proof}

We now consider the functions $f_{\bar g}$ on superbasic affine Deligne-Lusztig varieties. Our main result in this case is
\begin{prop}\label{propsuperb}
The decomposition of the affine Grassmannian, and of each affine Deligne-Lusztig variety according to weak EL-charts coincides with the decomposition by the $f_{\bar g}:J_{\bar{b}}(F)\rightarrow X_*(T')_{\dom}$ together with $\kappa_G(g)$.

In the case $G=\Res_{F'|F}GL_{h}$ there is the following more precise description. Let $g\cdot\Lambda_0$ be a lattice with weak EL chart $A=A^{\mu}$ where $A^{\mu}$ is defined in Remark ~\ref{group_semi_module} with $\mu\in X_*(T)$. Then $f_g(\pi^l)=\mathrm{inv}(\pi^l, \epsilon^\mu)$ for all $l\in\mathbb{Z}$.
Furthermore, these values for $l=1,\dotsc, h$ determine the function $f_g$.
\end{prop}
The last statement is in particular a stronger version of Proposition \ref{propfin} for this case.
\begin{proof}
By the reduction we gave above, the general assertion follows from the one for the special case $G=\Res_{F'|F}GL_{h}$. So from now on we assume that we are in this case.

First consider the second assertion. Note that $\pi I \pi^{-1}=I$ and $g\in I\epsilon^\mu K$ by definition. Hence $f_g(\pi^l)=\mathrm{inv}(\pi^l, \epsilon^\mu)$ for all $l\in\mathbb{Z}$. For the last assertion, by Lemma \ref{lemjsuperbasic} and the fact that $f_g$ is constant on $K\cap J_b(F)$-cosets, it is enough to consider the values on all $\pi^l$ for $l\in\mathbb Z$. On the other hand, $\pi^h$ maps $e_{\tau,i}$ to $e_{\tau,i+h}$, and hence coincides with the central scalar multiplication by $\e$. Thus the value of $f_g$ on $\pi^{l+h}$ is determined by that on $\pi^l$ (by adding the corresponding central cocharacter). The last assertion follows.

The first statement is implied by the second and third one and the fact that for any $\mu, \mu'\in X_*(T)$,
\[\mu=\mu'\Longleftrightarrow\mathrm{inv}(\pi^l, \epsilon^\mu)=\mathrm{inv}(\pi^l, \epsilon^{\mu'}), \text{ for all } l\in\mathbb{Z}.\]\end{proof}
\section{Relation to Ekedahl-Oort strata}
In this section we compare our stratification to the stratification of RZ spaces or affine Deligne-Lusztig varieties induced by the truncation of level 1, or equivalently, the EO-invariant. It turns out that both stratifications are incomparable in general. However, they are very closely related in several particularly interesting cases where the EO stratification is known to induce a paving of the affine Deligne-Lusztig variety by classical Deligne-Lusztig varieties.

\subsection{The general case}

Classically, and on moduli spaces of $p$-divisible groups, the EO invariant is defined as the isomorphism class of the $p$-torsion of the $p$-divisible groups with the induced additional structure. The corresponding group theoretic analog is called the truncation of level 1 and was introduced in \cite{trunc1}. It coincides with the EO invariant in the RZ cases (see the introduction of \cite{trunc1}) and is defined on all affine Deligne-Lusztig varieties. Let us briefly summarize its definition.

By definition the truncation of level 1 is the invariant given by the $K$-$\sigma$-conjugacy class of $K_1 h K_1$ for $h\in G(L)$. Here, $K_1$ is the inverse image of $1$ under the projection $K\rightarrow G(\overline{k})$. To recall the classification of truncations of level 1 fix some $h\in G(L)$. By the Cartan decomposition there is a unique dominant coweight $\mu$ such that $h\in K\mu(\e)K$. Let $W$ denote the (absolute) Weyl group of $T$ in $G$. For each $w\in W$ we choose a representative $w$ in $N_T(\mathcal{O})$ where $N_T$ is the normalizer of $T$ in $G$. Let $M_{\mu}$ be the centralizer of $\mu$ and let $^{\mu}W=\sigma^{-1}( {}^{M_{\mu}}W)$. If $M$ is a Levi subgroup of $G$ containing $T$ let $W_M$ be the Weyl group of $M$ and denote by ${}^M W$ the set of elements $x$ of $W$ that are shortest representatives of their coset $W_Mx$. Let $x_{\mu}=w_0w_{0,{\mu}}$ where $w_0$ denotes the longest element of $W$ and where $w_{0,{\mu}}$ is the longest element of $W_{M_\mu}$. Let $\tau_{\mu}=x_{\mu}\e^{\mu}$. Then $\tau_{\mu}$ is the shortest element of $W\epsilon^{\mu}W$. Then there is a unique $w\in  {}^{\mu}W$ such that $h$ is contained in the $K$-$\sigma$-conjugacy class of $K_1w\tau_{\mu}K_1$. The pair $(w,\mu)$ is called the truncation of level 1 of $h$.

To apply this to an affine Deligne-Lusztig variety $X_{\mu}(b)$, let $g\in X_{\mu}(b)$. Then the element $g^{-1}b\sigma(g)$ is well-defined up to $K$-$\sigma$-conjugation. Its (well-defined) truncation of level 1 is of the form $(w,\mu)$ where $\mu$ is the fixed cocharacter. Thus the truncation of level 1 defines a subdivision of $X_{\mu}(b)$ indexed by ${}^{\mu}W$ (or equivalently by $\sigma^{-1}(W_{M_{\mu}})\backslash W$). By \cite{trunc1},Theorem 1.4, it is a stratification.

A second, coarser and closely related stratification is the Bruhat stratification considered in \cite{WedhornBruhat}. Its definition in our language is as follows: Let $h\in G(L)$ and let $(w,\mu)$ be its truncation of level 1. Then its Bruhat invariant is the pair $(\overline w,\mu)$ where $\overline w$ is the class of $w$ in $\sigma^{-1}(W_{M_{\mu}})\backslash W/(x_{\mu}W_{M_{\mu}}x_{\mu}^{-1})$. Again, applying the Bruhat invariant to $g^{-1}b\sigma(g)$ induces a stratification of affine Deligne-Lusztig varieties.

For obvious reasons, the EO stratification and the Bruhat stratification cannot be finer than ours. Indeed, both of them are invariants computed via the elements $g^{-1}b\sigma(g)$ and are thus invariant under the $J_b(F)$-action. On the other hand, our invariant is designed in such a way that the strata are permuted by $J_b(F)$ and have bounded stabilizer groups. We expect that even if one would replace our invariant by the coarser one obtained by $J_b(F)$-orbits of the functions $f_g$, the EO stratification is still not finer.

\begin{ex}
We give an example which shows that also our stratification is not a refinement of the Bruhat stratification (and thus in particular not of the EO stratification). For this we consider the case $G=GL_9$, choose the Borel subgroup $B$ to consist of the lower triangular matrices and let $T$ be the diagonal torus. Let $[b]$ be superbasic of slope $4/9$ and $\mu$ minuscule with $[b]\in B(G,\mu)$, i.e. $\mu=(0^{(5)},1^{(4)})$. Let $e_1,\dotsc, e_9$ be the standard basis of $L^9$. Let $b\in [b]$ be the representative with $b\sigma(e_i)=e_{i+4}$ where $e_{i+9}=\e e_i$.

Let $A_1$ be the semi-module $A_1=\{1,2,5,6,7,\dotsc\}$. Then an explicit computation (using the algorithm to compute truncations of level 1 suggested by the proof of the classification theorem in \cite{trunc1}) gives the following: The lattice generated by all $e_i$ with $i\in A_1$ is in the semi-module stratum associated with $A_1$, and its Bruhat invariant is $(\overline{w_1},\mu)$ where $\mu$ is as above and where the shortest representative of $\overline{w_1}$ is the three-cycle $(465)\in S_9=W$. On the other hand consider the lattice generated by $e_1+ce_3, e_2+de_4$ and all $e_i$ with $i\geq 5$. It has the same semi-module. However, if $\sigma^2(d)\neq c$, then its Bruhat invariant is $(\overline{w_1}',\mu)$ where $\mu$ is as above and where the shortest representative of $\overline{w_1}'$ is $(36475)\in S_9=W$.
\end{ex}

\subsection{The Bruhat-Tits stratification}\label{sec42}

There are a small number of cases where the connected components of Ekedahl-Oort strata are irreducible, and isomorphic to Deligne-Lusztig varieties. The first larger class of such affine Deligne-Lusztig varieties was studied by Vollaard and Wedhorn, and occurs for unitary groups of signature $(1,n-1)$. We discuss it in more detail in Section \ref{secVW}. In \cite{GH}, G\"ortz and He associate with an affine Deligne-Lusztig variety the triple of the group $G$, the Hodge weight, and the associated parahoric subgroup (in our case the group $K$). Such a triple determines a unique basic $\sigma$-conjugacy class $[b]\in B(G,\mu)$. They define a group-theoretic property to be of Coxeter type which turns out to imply that Vollaard-Wedhorn's theory can be generalized to the corresponding basic affine Deligne-Lusztig variety. They give a complete list of 21 cases of triples of Kottwitz type and describe the stratifications of the corresponding basic affine Deligne-Lusztig varieties.

\begin{conj}
If $(G,\mu, K)$ is of Coxeter type, then the paving of the corresponding basic affine Deligne-Lusztig variety by classical Deligne-Lusztig varieties induced by the Ekedahl-Oort stratification coincides with the decomposition according to the functions $f_g$.
\end{conj}
We also expect that a more general version of this conjecture holds, where one replaces $K$ by a general parahoric $P$ such that $(G,\mu,P)$ is of Coxeter type, and uses the obvious generalization of our invariant. Below we verify our conjecture in two important example cases where an explicit description of the stratification is available in the literature.

\subsubsection{The affine Deligne-Lusztig variety associated with the supersingular locus in the moduli space of abelian surfaces}
We consider $G=GSp_4$, $[b]$ basic of slope $\frac{1}{2}$, and $\mu=(1,1,0,0)$ the unique minuscule element with $b\in B(G,\mu)$. Then points of the corresponding affine Deligne-Lusztig variety correspond to  Dieudonn\'e lattices in a polarized supersingular isocrystal of dimension 4 which are self-dual up to a factor. Here a lattice is called a Dieudonn\'e lattice if it is invariant under $\Phi=b\sigma$ and under $V=\e\Phi^{-1}$. For $F=\Q_p$ this affine Deligne-Lusztig variety can be used to study the supersingular locus in the Siegel moduli space of principally polarized abelian surfaces. Its geometry has been described by many people, Moret-Bailly \cite{MoBa} and Oort \cite{OortK}, and later Kaiser \cite{Kai}, and Kudla and Rapoport \cite{KudlaRapo}. We summarize the results in a way enabling us to relate them to our theory. We continue to consider both the function field case and the arithmetic case.

Let $(N,\langle,\rangle, F)$ be a polarized supersingular isocrystal of dimension 4, i.e. with $N\cong L^4$. We can choose a basis $e_1,e_2,f_1,f_2$ of $N$ in such a way that $F(e_i)=e_{i+1}$ (with $e_2=\e e_0$) and with analogous notation $F(f_i)=f_{i+1}$, and such that $\langle e_1, f_2\rangle=\langle e_2, f_1\rangle=1$ and all other pairings of two of the basis vectors are 0.

Let $\Lambda_0$ be the standard lattice in $N$ generated by our chosen basis. Then $g\mapsto g\Lambda_0$ induces a bijection between $X_{\mu}(b)$ and the set of Dieudonn\'e lattices $M$ in $N$ such that $\langle\cdot,\cdot\rangle|_M$ is perfect up to a factor. The classification by Kaiser subdivides the set of these lattices into two subsets. One is the $J_b(F)$-orbit of $\Lambda_0$ itself, that is the set of all Dieudonn\'e lattices defined over $\mathbb{Q}_{p^2}$ and self-dual up to a factor with respect to $\langle \cdot,\cdot\rangle$. For the other subset we first consider a family of such lattices parametrized by $\mathbb{P}_{\mathbb{F}_p}^1\setminus\mathbb{P}_{\mathbb{F}_p}^1(\mathbb{F}_{p^2})$ mapping a point $x$ to the lattice generated by $e_1+[x]f_0,e_2,f_1,f_2$ where $f_0=\e^{-1}f_2$ and where $[x]$ is the Teichm\"uller representative of $x$. Then the second kind of lattices is the $J_b(F)$-orbit of this family. One obtains a stratification of $X_{\mu}(b)$ whose strata are either a single point of the form $j\Lambda_0$ for some $j\in J_b(F)$  or some subset $j\{\langle e_1+[x]f_0,e_2,f_1,f_2\rangle\mid x\in \mathbb{P}_{\mathbb{F}_p}^1\setminus \mathbb{P}_{\mathbb{F}_p}^1(\mathbb{F}_{p^2}) \}$, and thus isomorphic to $\mathbb{P}_{\mathbb{F}_p}^1\setminus \mathbb{P}_{\mathbb{F}_p}^1(\mathbb{F}_{p^2})$. One could also characterize these strata by the property that they are either a single lattice of the form $j\Lambda_0$, or contained with corank 1 in a lattice of the form $j\langle e_1,e_2,f_0,f_1\rangle$, which then determines the stratum. The closure of a one-dimensional stratum is isomorphic to $\mathbb{P}^1$, the complement of its open stratum consists of $p^2+1$ of the zero-dimensional strata. Each $0$-dimensional stratum lies in the closure of exactly $p+1$ strata of dimension 1. More precisely, the incidence graph of this configuration of projective lines is described by a union of copies of the Bruhat-Tits building of $J_b(F)$.

Let us now compute the functions $f_g$ for the elements $g\in X_{\mu}(b)$. For $g$ corresponding to a lattice $\Lambda$ of the first kind we may choose a representative of $gK$ such that $g\in J_b(F)$. Then $f_g(g)$ is trivial and this condition determines $f_g$. For $g$ of the second kind $gK$ does not have a representative in $J_b(F)$, so in particular, $f_g$ does not take the value $0$. To compute $f_g$ for this case we assume (by modifying $g$ by an element of $J_b(F)$) that the corresponding lattice is generated by $e_1+[x]f_0,e_2,f_1,f_2$ with $x$ as above. Then $f_g$ takes its minimal possible value $(1,0,0,-1)$ precisely on those $j\in J_b(F)$ with $j(\Lambda_0)$ generated by $[a]e_1+[b]f_0, e_2,f_1,f_2$ with $(a:b)\in \mathbb{P}_{\mathbb{F}_p}^1(\mathbb{F}_{p^2})$. The union of these lattices is the lattice generated by $e_1,e_2,f_0,f_1$ which determines the stratum containing $e_1+[x]f_0,e_2,f_1,f_2$. Conversely, using an argument as in the proof of the claim of Section \ref{secclos}, one can see that the function $f_g$ for $g\Lambda_0=\langle e_1+[x]f_0,e_2,f_1,f_2\rangle$ is determined by the dimensions of the intersections of $g\Lambda_0/\langle e_2,e_3,f_1,f_2\rangle$ (which is generated by $e_1+[x]f_0$) with all sub-vector spaces of $\langle e_1,e_2,f_0,f_1\rangle/\langle e_2,e_3,f_1,f_2\rangle$ which are defined over $\mathbb{F}_{p^2}$. In particular, it is independent of $x$ as long as $x\notin\mathbb{F}_{p^2}$. Altogether we have shown the following proposition.

\begin{prop}\label{propkaiser}
The above stratification of the moduli space of two-dimensional principally polarized $p$-divisible groups coincides with the stratification by the functions $f_g$.
\end{prop}

\subsubsection{The Vollaard-Wedhorn case}\label{secVW}

In \cite{VW}, Vollaard and Wedhorn consider the supersingular locus of the Shimura variety for a group $G$ such that $G_{\mathbb{R}}\cong GU(1,n-1)$ is the unitary similitude group of signature $(1,n-1)$, and at a prime $p$ which is inert in the reflex field of the Shimura datum and such that $G$ is unramified at $p$.

Using Rapoport-Zink uniformization, the supersingular locus can be studied by considering the corresponding moduli space of $p$-divisible groups with PEL structure. Its underlying reduced subscheme is an affine Deligne-Lusztig variety for the group $G_{\Q_{p^2}}$ (note that $\Q_{p^2}$, the unramified extension of $\mathbb{Q}_p$ of degree 2, is the completion of the reflex field $E$ at the unique prime lying over $p$). It has the following explicit description, compare \cite{VW}, Section 1, where also a more conceptual definition and the relation to the moduli space can be found. Below we consider again both the function field case and the arithmetic case. The case of \cite{VW} corresponds to choosing $F=\Q_p$.

Let $K$ be an unramified extension of $F$ of degree 2, and $\sigma$ the nontrivial automorphism of $K$ over $F$. Let $\O_K$ be its ring of integers. Then we define a unitary Dieudonn\'e module as follows. Let $\delta\in \O_K^{\times}$ with $\sigma(\delta)=-\delta$. Let $\tilde {\mathbb{S}}=\O_K\otimes_{\O_F}\O_K=\O_K\oplus\O_K$ (decomposition with respect to the two embeddings into an algebraic closure of $F$) and let $g=(1,0)$ and $h=(0,1)\in \tilde {\mathbb{S}}$. We define the $\sigma$-linear Frobenius map $F$ on $\tilde {\mathbb{S}}$ by $F(g)=\e h$ and $F(h)=g$, and a pairing $\langle\cdot,\cdot\rangle$ by $\langle g,h\rangle=\delta$. Let $N_{(0,1)}=\tilde{\mathbb{S}}\otimes_{\O_K}K$ and let $N=N_{(0,1)}^{n-1}\oplus \sigma^*(N_{(0,1)})$ where $\sigma^*(N_{(0,1)})$ is the same as $N_{(0,1)}$ as symplectic $\O_K\otimes\O_K$-module, but with Frobenius map $F(g)=h$ and $F(h)=\e g$. Then $N$ is a supersingular symplectic isocrystal, and $M'=\tilde{\mathbb{S}}^{n-1}\oplus \sigma^*(\tilde{\mathbb{S}})$ is a unitary Dieudonn\'e module of signature $(1,n-1)$.Consider the decomposition of $N$ into $N_0\oplus N_1$, the subspaces generated by the basis vectors $g$ resp.~$h$ in each summand. The set of lattices $M\subseteq N\otimes\widehat{K^{nr}}$ with $\e M\subseteq F(M)\subseteq M$, self-dual up to a factor $p^i$, with $M=M_0\oplus M_1$ where $M_i=M\cap N_i$, and satisfying a signature condition (detailed below) can then be identified with the set of $\overline {\mathbb{F}}_p$-valued points of a corresponding moduli space of $p$-divisible groups, or with an affine Deligne-Lusztig set associated with a minuscule coweight $\mu$.

Define a $\sigma$-hermitian form on $N_0$ by $\{x,y\}=\delta\langle x,Fy\rangle$. Then Vollaard and Wedhorn show that the set of lattices $M$ as above can be described as follows: The summand $M_0$ determines $M=M_0\oplus M_1$ uniquely. For a lattice $M_0\subset N_0$ we define $M_0^{\vee}$ to be the dual of $M_0$ with respect to $\{\cdot,\cdot\}$. Then the set of lattices $M$ as above (i.e.~also satisfying the signature condition) is in bijection with the set of lattices $M_0$ in $N_0$ such that
\begin{equation}\label{eqbtvw}
\e^{i+1} M_0^{\vee}\stackrel{1}{\subset}M_0\stackrel{n-1}{\subset}\e^iM_0^{\vee}
\end{equation}
for some $i$. In the study of this moduli space, one may restrict one's attention to one of the isomorphic connected components, and therefore assume that $i=0$. Let $\tau=\e^{-1}F^2$. Vollaard and Wedhorn decompose the (component of the) moduli space according to the
invariant assigning to such a lattice $M_0$ the lattice $\Lambda=M_0+
\tau(M_0)+\dotsm+\tau^{d}M_0$ where $d$ is minimal such that this
lattice is  $\tau$-invariant. It is the smallest $\tau$-invariant lattice containing $M_0$ and satisfies (see \cite{Vollaard}, Lemma 2.2)
\begin{equation}\label{typelambda}
\e\Lambda^{\vee}\stackrel{2d+1}{\subset}\Lambda\stackrel{n-2d-1}{\subset}\Lambda^{\vee}.
\end{equation}

On the other hand, the explicit description of the group $J_b(F)$ for this
case (as given in \cite{VW}, 1.6) shows that our invariant, the
functions $f_g$, correspond to fixing the relative positions
$\inv(M_0,\Lambda')$ for all $\tau$-invariant lattices
$\Lambda'\subseteq N_0$ with $\e^i(\Lambda')^{\vee}\stackrel{1}{\subset}\Lambda'\stackrel{n-1}{\subset}\e^{i-1}(\Lambda')^{\vee}$ for some $i$.

We now show that these two invariants coincide.

\begin{thm}
Two lattices $M_0, M_0'$ with (\ref{eqbtvw}) have the same
Vollaard-Wedhorn invariant $(\Lambda,d)$ if and only if
$\inv(M_0,\Lambda')=\inv(M'_0,\Lambda')$ for all $\tau$-invariant
lattices $\Lambda'\subseteq N_0$ with $\e^i(\Lambda')^{\vee}\stackrel{1}{\subset}\Lambda'\stackrel{n-1}{\subset}\e^{i-1}(\Lambda')^{\vee}$ for some $i$.
\end{thm}
\begin{proof}
We first assume that the Vollaard-Wedhorn invariant of $M_0$ and $M_0'$
is equal to the same pair $(\Lambda,d)$. We claim that then even
$\inv(M_0,\Lambda')=\inv(M'_0,\Lambda')$ for all $\tau$-invariant
lattices $\Lambda'\subseteq N_0$. Let $\Lambda'$ be such a lattice. We
consider for every $i\in\mathbb Z$ and all lattices $\tilde M_0$ in
$N_0$ the volume $v(\tilde M_0,i):=\vol(\tilde M_0+\e^i\Lambda')$.
Together, the $v(\tilde M_0,i)$ determine $\inv (\tilde M_0,\Lambda')$,
so it is enough to show that for fixed $i$ they are constant for $\tilde
M_0$ in the given stratum. Each stratum of \cite{VW} is irreducible. From the definition we obtain that in the
generic point of the stratum, $v(\cdot ,i)$ is less or equal to the value
in any other point of the stratum. On the other hand for $\tilde{M}_0$ in the
given stratum let $n_{\tilde{M}_0,i}\leq d$ be minimal such that
$\tilde{M}_0+\dotsm+\tau^{n_{\tilde{M}_0,i}}\tilde{M}_0+\e^i\Lambda'$ is $\tau$-invariant. By
\cite{Vollaard}, Lemma 2.2 and its proof, we have $\vol(\tilde{M}_0+\dotsm+\tau^l\tilde{M}_0)\geq \vol(\tilde{M}_0)-l$ with
equality for all $l\leq d$. Thus from the minimality of $n_{\tilde{M}_0,i}$ we
obtain that $$v(\Lambda,i)=v(\tilde{M}_0+\dotsm+\tau^{n_{\tilde{M}_0,i}}\tilde{M}_0+\e^i\Lambda',i)=
v(\tilde{M}_0,i)-n_{\tilde{M}_0,i}.$$ The left hand side of this equation is constant on
the given stratum. The constant $n_{\tilde{M}_0,i}$ decreases under
specialization, which completes the proof that $v(\tilde{M}_0,i)$ is also
constant on the given stratum.

Let us now show how to compute $(\Lambda,d)$ from the values
$\inv(M_0,\Lambda')$ for all $\tau$-invariant
lattices $\Lambda'\subseteq N_0$ with $\e^i(\Lambda')^{\vee}\stackrel{1}{\subset}\Lambda'\stackrel{n-1}{\subset}\e^{i-1}(\Lambda')^{\vee}.$ Inspired by (\ref{typelambda}) call a $\tau$-invariant lattice $\Lambda_j$ with $$\e^i(\Lambda_j)^{\vee}\stackrel{j}{\subset}\Lambda_j\stackrel{n-j}{\subset}\e^{i-1}(\Lambda_j)^{\vee}$$ for some $i$ and some odd $j$ a vertex lattice (of type $(j;i)$). Then computing the invariant $(\Lambda,d)$ is equivalent to determining for each vertex lattice $\Lambda_j$ of some type $(j;1)$ if $M_0\subseteq \Lambda_j$. From our invariant we already know if $M_0\subseteq \Lambda_1$ for vertex lattices of types $(1;i)$ for all $i$.

\underline{Case 1: $n$ is odd.}

We use decreasing induction on $j$ to show that our invariant determines if $M_0$ is contained in a given vertex lattice $\Lambda_j$ of type $(j;1)$.

For $j=n$, consider all $\tilde{\Lambda}_1$ of type $(1;-1)$ satisfying
\[\Lambda_n=\Lambda_n^{\vee}\subset \e^{-1}\tilde{\Lambda}_1^{\vee}\subset\tilde{\Lambda}_1\subset \e^{-1}\Lambda_n.\]
By the corollary below, the intersection of these $\tilde{\Lambda}_1$ is equal to $\Lambda_n$. In particular, $M_0\subset \Lambda_n$ if and only if $M_0\subset \tilde\Lambda_1$ for each of these $\tilde{\Lambda}_1,$ a condition that is determined by our invariant.

For the induction step suppose that we can already test whether $M\subset \Lambda_{j+2}$ with $\Lambda_{j+2}$ of type $(j+2;1)$. Now we want to test with $\Lambda_j$ of type $(j;1)$. We consider all $\tilde{\Lambda}_{j+2}$ of type $(j+2;1)$ such that
\[\Lambda_j\subset \tilde{\Lambda}_{j+2}\subset \tilde{\Lambda}_{j+2}^{\vee}\subset \Lambda_j^{\vee}. \]
By the lemma below, their intersection is $\Lambda_j$. By the  induction hypothesis, we know whether $M_0$ is contained in $\tilde{\Lambda}_{j+2}$, and hence also whether $M_0$ is contained in $\Lambda_j$.

\underline{Case 2: $n$ is even.}

Let $\Lambda_j$ be of type $(j;1)$ with $j<n-1$ and odd. Then $\Lambda_j^{\vee}$ is of type $(n-j;0)$. We consider all lattices $\Lambda_1$ of type $(1;0)$ with $$\Lambda_j\stackrel{(n-j-1)/2}{\subset}\Lambda_1^{\vee}\stackrel{1}{\subset}\Lambda_1\stackrel{(n-j-1)/2}{\subset}\Lambda_j^{\vee}.$$

By the corollary below (for $l=0$) the intersection of these is $\Lambda_j$. So our invariant determines if $M_0$ is contained in $\Lambda_j$ (for $j<n-1$).

It remains to test whether $M_0\subset \Lambda_{n-1}$ for vertex lattices $\Lambda_{n-1}$ of type $(n-1;1)$.

Claim: $M_0\subset \Lambda_{n-1}\Longleftrightarrow \e\Lambda^{\vee}_{n-1}\stackrel{\frac{n}{2}}{\subset} M_0\stackrel{\frac{n}{2}}{\subset} \Lambda_{n-1}^{\vee}$.

As $\Lambda_{n-1}^{\vee}$ is of type $(1;0)$, our invariant determines whether the right hand side is true or not. It remains to prove the claim. An easy calculation shows that the first condition implies the second. Now we show the other direction. Suppose the second condition holds. By duality this is equivalent to the condition
$$\e\Lambda_{n-1}\stackrel{\frac{n}{2}}{\subset} \e M_0^{\vee}\stackrel{\frac{n}{2}}{\subset} \Lambda_{n-1}.$$ The hermitian form $\{\cdot, \cdot\}$ on $\Lambda_{n-1}$ induces an hermitian form $$\{\cdot,\cdot\}: \Lambda_{n-1}/\e\Lambda_{n-1}\times \Lambda_{n-1}/\e\Lambda_{n-1}\rightarrow \mathbb{F}_{p^2}.$$ Moreover, we have an isomorphism of hermitian forms

\[\Lambda_{n-1}/\e\Lambda_{n-1}\simeq \Lambda_{n-1}/\e\Lambda_{n-1}^{\vee}\oplus \e\Lambda_{n-1}^{\vee}/\e\Lambda_{n-1}\]
where the hermitian form on $\Lambda_{n-1}/\e\Lambda_{n-1}^{\vee}$ is induced by $\{\cdot, \cdot\}$ which is non-degenerate and the hermitian form on $\e\Lambda_{n-1}^\vee/\e\Lambda_{n-1}$ is the zero form. The dimensions of the two summands are $n-1$ and 1. Consider the $\mathbb{F}_{p^2}$-subspace $\e M_0^{\vee}/\e\Lambda_{n-1}$ of $\Lambda_{n-1}/\e\Lambda_{n-1}$. It is totally isotropic of dimension $\frac{n}{2}$. Therefore $\e M_0^{\vee}/\e\Lambda_{n-1}\supset \e\Lambda_{n-1}^{\vee}/\e\Lambda_{n-1}$ (as the totally isotropic subspaces in $\Lambda_{n-1}/\e\Lambda_{n-1}^{\vee}$ are of dimension at most $\frac{n-2}{2}$). Hence $\e\Lambda_{n-1}^{\vee}\subset \e M_0^{\vee}$ and by duality we have $M_0\subset \Lambda_{n-1}$. This proves the claim.
\end{proof}
\begin{lemma} Let $3\leq m \leq n$ be an odd integer. For any $\Lambda_m$ of type $(m;1)$ consider all $\Lambda_{m-2}$ of type $(m-2;1)$ such that \[\e\Lambda_m^{\vee}\stackrel{1}{\subset}\e\Lambda_{m-2}^{\vee}\stackrel{m-2}{\subset}\Lambda_{m-2}\stackrel{1}{\subset}\Lambda_m.\] Then the intersection of these $\Lambda_{m-2}$ is $\e\Lambda_{m}^{\vee}$.
\end{lemma}
\begin{proof} Let $U=\Lambda_m/\e\Lambda_{m}^{\vee}$, a $\mathbb{F}_{p^2}$-vector space of dimension $m$. The hermitian form $\{\cdot, \cdot\}$ induces a non-degenerate hermitian form on $U$. As all hermitian forms over $\mathbb{F}_{p^2}$ are isomorphic, after choosing a basis of $U$, we may assume that $U\simeq \mathbb{F}_{p^2}^{m}$ and that the hermitian form on $U$ is given by  $$\{(x_1,\dotsc, x_m), (y_1,\dotsc, y_m)\}= x_1y_1^p+\dotsc x_my_m^p.$$ Let $\Theta_m$ be the number of lattices $\Lambda_{m-2}$ as above. As they are determined by the (1-dimensional) quotient $\e\Lambda_{m-2}^{\vee}/\e\Lambda_m^{\vee}$, we have
$\Theta_m=|M_m|$ for $$M_m=\{[x_1:\dotsc: x_m]\in\mathbb{P}^{m-1}(\mathbb{F}_{p^2})| x_1^{p+1}+\cdots+x_m^{p+1}=0\}.$$ We want to compute this value, and to do so we consider it for all values of $m$, not only odd ones. Note that the map $\phi:\mathbb{F}_{p^2}\rightarrow \mathbb{F}_{p^2}$ with $x\mapsto x^{p+1}$ has image in $\mathbb{F}_p$, the inverse image of $0$ consists of $0$, and all other nonempty fibers consist of $p+1$ elements. Thus to determine $\Theta_m$ we decompose $M_m$ according to whether $x_m=0$ or not. In the first case we have that $[x_1:\dotsc: x_{m-1}]\in M_{m-1}$. In the second, the first $m-1$ coordinates can be chosen in such a way that $[x_1:\dotsc: x_{m-1}]\notin M_{m-1}$, and by our discussion of the fibers of $\phi$ there are then $p+1$ possible values for $x_m$. Altogether we obtain $$\Theta_m=\Theta_{m-1}+(|\mathbb{P}^{m-2}(\mathbb{F}_{p^2})|-\Theta_{m-1})(p+1)=\frac{(p+1)(p^{2m-2}-1)}{p^2-1}-p\Theta_{m-1}.$$
Further, $\Theta_2=p+1.$ An easy inductive argument then shows that
$$\Theta_m=\frac{(p^m-(-1)^m)(p^{m-1}+(-1)^{m})}{p^2-1}.$$

If the intersection of $\Lambda_{m-2}/\e\Lambda_m^{\vee}\subset U$ for all $\Lambda_{m-2}$ is not trivial, then all $\Lambda_{m-2}/\e\Lambda_m^{\vee}$ contain a common 1-dimensional sub-vector space in $U$. Consider the number $\Psi_m$ of all sub-$\mathbb{F}_{p^2}$-vector spaces in $U$ of dimension $m-1$ containing a fixed 1-dimensional sub-vector space. Then \[\Psi_m=|\mathbb{P}^{m-2}(\mathbb{F}_{p^2})|=\frac{p^{2m-2}-1}{p^2-1}=\frac{(p^{m-1}-1)(p^{m-1}+1)}{p^2-1}.\] As $\Theta_m>\Psi_m$, the lemma follows.
\end{proof}

\begin{kor}
Let $3\leq m \leq n$ be an odd integer and let $0<i<m/2$. For any $\Lambda_m$ of type $(m;l)$ with $l\in\{0,1\}$, consider all $\Lambda_{m-2i}$ of type $(m-2i;l)$ such that \[\e^l\Lambda_m^{\vee}\stackrel{i}{\subset}\e^l\Lambda_{m-2i}^{\vee}\stackrel{m-2i}{\subset}\Lambda_{m-2i}\stackrel{i}{\subset}\Lambda_m.\] Then the intersection of these $\Lambda_{m-2i}$ is $\e^l\Lambda_{m}^{\vee}$.
\end{kor}
\begin{proof}
For $l=1$, this follows from the lemma using induction on $i$. For $l=0$ apply the assertion for $l=1$ to the hermitian form $\e^{-1}\{\cdot,\cdot\}$.
\end{proof}



\section{Relation to the $a$-invariant}

An important invariant of a $p$-divisible group $X$ over $\bar k$ is the dimension of $\Hom_{\bar k}(\alpha_p,X)$, called the $a$-number $a(X)$. For example one can show that on a Rapoport-Zink moduli space of $p$-divisible groups without imposed endomorphisms, generically the $a$-number of the $p$-divisible groups is 1 unless we are in the degenerate case of ordinary $p$-divisible groups (see \cite{modpdiv}, \cite{polpdiv}, building on results of Oort on deformations of $p$-divisible groups).

Translated into our context of affine Deligne-Lusztig varieties this invariant yields the following. Consider the group $G=GL_h$. (Generalizations for example to $GSp_{2h}$ are possible, but are only induced by viewing these groups as subgroups of some $GL_h$, so are less interesting). Let $B$ be the Borel subgroup of upper triangular matrices and let $T$ be the diagonal torus. Let $\mu=(1,\dotsc, 1,0,\dotsc, 0)\in \Z^h_+\cong X_*(T)_{\dom}$ with multiplicities $m,h-m$. Let $[b]\in B(G,\mu)$. Then the geometric points of the associated affine Deligne-Lusztig variety correspond to lattices $M$ in the isocrystal $(L^h,b\sigma)$ satisfying $M\supseteq b\sigma(M)\supseteq \e M$. In other words, the lattices is supposed to be stable under $\Phi=b\sigma$ and under $V=\epsilon(b\sigma)^{-1}$. We call such lattices Dieudonn\'e lattices. The $a$-invariant of such a lattice is then defined as $\rg(M/(\Phi(M)+V(M)))$. Assigning to a $p$-divisible group its Dieudonn\'e module gives back the original $a$-number.

One can easily find examples proving that in general our invariant is not comparable to the $a$-invariant. However, for the most important stratum defined by the $a$-invariant (namely the stratum which is in all non-degenerate cases generic in the affine Deligne-Lusztig variety), there is the following comparison.

\begin{conj}\label{conjaf}
Let $(N,b\sigma)$ be an isocrystal. Then there is a function $f:J_b(F)\rightarrow X_*(T)_{\dom}$ such that a Dieudonn\'e lattice $M=gM_0$ satisfies $a(M)=1$ if and only if $f_g$ is in the $J_b(F)$-orbit of $f$.
\end{conj}
In other words, the locus where $a=1$ in the affine Deligne-Lusztig variety for $G=\GL_h$, given $b$ and minuscule $\mu$ is equal to one $J_b(F)$-orbit of strata for our invariant.

To illustrate this conjecture, and to prove an important special case, we need some notation.

Let $N=\oplus_{z=1}^{z_0}N_{\lambda_z}$ be the isotypic decomposition of $N$ with $N_{\lambda_z}$ isoclinic of slope $\lambda_z$ with $\lambda_z<\lambda_{z'}$ for $z<z'$. Let $N_{\lambda_z}=\oplus_{i=1}^{l_z}N_{z, i}$ be a decomposition into simple isocrystals. Let $\lambda_z=m_z/h_z$ for all $1\leq z\leq z_0$ with $m_z$ and $h_z$ coprime natural numbers. Let $n_z=h_z-m_z$. For any $z$, let $\pi_z:=\e^{a_z}(b\sigma)^{a'_z}\in D_{\lambda_z}=\mathrm{Aut}(N_{z,i})$ where $D_{\lambda_z}$ is the division algebra of invariant $\lambda_z$ over $F$ and $a_z\in\Z$ is minimal positive such that there is an $a_z'\in\Z$ with $a'_zh_z+a_zm_z=1$. Then $\pi_z$ is a uniformizer of $D_{\lambda_z}$ and $J_b(F)=\prod_{z=1}^{z_0}GL_{l_z}(D_{\lambda_z})$ (see \cite{deJongOort}). Choose $e_{zi0}\in N_{z,i}^{(b\sigma)^{h_z}=\e^{m_z}}\backslash\{0\}$ and let $e_{zil}:=\pi_z^l e_{zi0}$ for all $l\in\Z$. Note that then $e_{z,i,l+h_z}=\e e_{zil}$ for all $z,i,l$. Let $M_0\subseteq N$ be the lattice generated by the $e_{zil}$ for all $z,i, l$ with $l\geq 0$.

\begin{definition}A {\it multi-semimodule} of type $N$ is a subset $A$ of $(\Z\cup\{\infty\})^{\sum_{z}l_z}$ such that the following conditions are satisfied:
\begin{enumerate}
           \item There is an $N_1\in\Z$ such that $(\gamma_{zi})\in A$ implies that $\gamma_{zi}\geq N_1$ for all $z,i$. Further there is an $N_2\in \Z$ such that every $(\gamma_{zi})\in (\Z\cup\{\infty\})^{\sum_{z}l_z}$ with $\gamma_{zi}\geq N_2$ for all $z,i$, but not all equal to $\infty$, is contained in $A$.
           \item  If $\underline{\gamma}=(\gamma_{zi})\in A$,  then not all the cordinates $\gamma_{zi}$ are  $\infty$ and for all $a,b\in\mathbb{N}$, $\underline{\gamma}+a\underline{m}+b\underline{n}\in A$ where $\underline{m}=(m_z)_{z,i}$, $\underline{n}=(n_z)_{z,i}$ and where addition is componentwise.
            \item  If $(\gamma_{zi}),(\gamma'_{zi})\in A$, then $$\min((\gamma_{zi}),(\gamma'_{zi})):=(\min\{{\gamma}_{zi},\gamma'_{zi}\})_{z,i}\in A.$$ Here we define $\min(\infty, a)=a$ and $\infty+a=\infty$ for all $a\in \Z\cup\{\infty\}$.
           \end{enumerate}
 \end{definition}

\begin{remark}If $N$ is a simple isocrystal, then a multi-semimodule of type $N$ is a semimodule in the sense of de Jong and Oort \cite{deJongOort}.
\end{remark}

For any non-zero vector $v\in N$ we can write $v$ as an infinite converging sum $v=\sum_{z,i,l}[a_{zil}]e_{zil}$ with $[a_{zil}]$ the representatives in $\O_L$ of elements $a_{zil}\in\overline k$. Let $$\underline{\gamma}(v):=(\gamma_{zi})_{z,i}\in(\Z\cup\{\infty\})^{\sum_{z}l_z}$$ with
\[\gamma_{zi}:=\begin{cases} \infty, &\text{if } a_{zil}=0 \text{ for all } l\\ \min\{l| a_{zil}\neq 0\}, &\text{otherwise}.\end{cases}\] For a Dieudonn\'e lattice $M\subseteq N$, let
\[A(M):=\{\underline{\gamma}(v)|v\in M, v\neq 0\}.\]
Then it is easy to check that $A(M)$ is a multi-semimodule of type $N$.

Note that the multi-semimodule is not a canonical invariant of lattices, as it depends on the decomposition of $N$ into simple isocrystals (and the choice of a basis in each summand, which is only unique up to shifts of the indices). The following proposition compares this to our invariant (but is not used in the sequel except for the corollary below).

\begin{lemma}\label{remark_inv_multi-sm} For Dieudonn\'e lattices $M=gM_0, M'=g'M_0\subseteq N$, we have $f_g=f_{g'}$ if and only if
$A(M)=A(M')$ for all decompositions of $N$ into simple isocrystals, and choices of bases $e_{zil}$ as above. Equivalently, it is enough to choose for every decomposition one fixed such basis, and to know the corresponding multi-semimodules for these.
\end{lemma}
For $\underline{u}=(u_{zi})_{z,i}\in\Z^{\sum_z l_z}$, let $$\pi^{\underline{u}}:=\mathrm{diag}((\pi_z^{u_{zi}})_{zi})\in J_b(F).$$
\begin{proof}
For some fixed decomposition, associated basis and the associated multi-semi\-modules, we have $A(M)=A(M')$ if and only if $\vol (M\cap \pi^{\underline{u}}M_0)=\vol (M'\cap \pi^{\underline{u}}M_0)$ for all $\underline{u}\in\Z^{\sum_z l_z}.$ In particular, by Example ~\ref{rmk_description_invariant_GLn}, if $f_g=f_{g'}$ for $g, g'\in X^G_{\mu}(b)$, then $A(gM_0)=A(g'M_0)$. For the other direction recall that in order to compute $f_g(j)$, it is enough to compute $\vol(j^{-1}gM_0\cap \e^iM_0)$ for all $i$. We have $\vol(j^{-1}gM_0\cap \e^iM_0)=\vol(gM_0\cap \e^ijM_0)-v_{\e}(\det j)$. The last expression can be computed (using the above) from the multi-semimodule of $gM_0$ with respect to the decomposition of $N$ and the basis that is obtained from $\{e_{zi}\mid 1\leq z\leq z_0, 1\leq i\leq l_z \}$ by applying $\e^ij$.

Finally, we want to show the multi-semimodule of a lattice $M$ obtained by keeping the same decomposition of $N$ but changing the basis can be computed from the original one. Such a base change amounts to a base change by an element of $J_b(F)$ fixing each summand in the decomposition of $N$ into simple isocrystals, and is thus given by an element of the form $\pi^{\underline{u}}$ for some $\underline{u}$. By what we saw in the first part of the proof, this implies that each of the multi-semimodules of $M$ then determines the other.
\end{proof}

We assume again that all multi-semimodules are computed with respect to some chosen decomposition and basis of $N$.
\begin{cor}\label{koracontained}
Let $M\subseteq M'$ be lattices with $A(M)=A(M')$. Then $M=M'$.
\end{cor}
\begin{proof}
This follows as $A(M)=A(M')$ implies in particular that $\vol (M\cap \e^iM_0)=\vol (M'\cap \e^iM_0)$ for all $i\in\Z$, and hence that $\vol~ M=\vol~M'$.
\end{proof}

We now use these general reformulations that allow to express our invariants in terms of multi-semimodules to prove Conjecture \ref{conjaf} for the case that the isocrystal $N$ is isoclinic. Note that for several reasons this case is a particularly interesting one. In the context of Shimura varieties, isoclinic isocrystals correspond to the unique closed (i.e.~basic) Newton strata. These are the only ones where the Newton stratum is in direct relation to the corresponding Rapoport-Zink moduli space of $p$-divisible groups, and also the Newton strata which is conjecturally the most interesting from a representation-theoretic point of view. For the rest of this section we assume that we are in this case, i.e.~that $z_0=1$. To ease the notation we will from now on leave out all indices $z=1$ and write $l=l_1, m,n, h$ etc.

\begin{definition} Let $S_l$ be the symmetric group of $l$ elements.  Let  $A_{\gen}$ be the multi-semimodule for $N$ generated by $\underline{0}=(0,\dotsc, 0)$ and all elements of the form $$\tau(\underbrace{\infty, \dotsc,\infty}_i, imn, \dotsc, imn)$$ for $1\leq i\leq l-1, \tau\in S_{l}.$ In other words, $A_{\gen}$ is defined as the (unique!) smallest multi-semimodule containing all of these elements.
\end{definition}

\begin{lemma}\label{lemma_Agen} For $\gamma=(\gamma_1,\dotsc, \gamma_l)\in (\Z\cup\{\infty\})^{l}\backslash\{\underline{\infty}\}$, let \[c_i(\gamma):=\sharp\{j|1\leq j\leq l, \gamma_j>\gamma_i\},\text{ for all }1\leq i\leq l.\] Then $\gamma\in A_{\gen}$ if and only if
\begin{eqnarray}\label{eqn_ci_lambda}\gamma_i\in (c_i(\gamma)mn+m\mathbb{N}+n\mathbb{N})\cup\{\infty\}, \text{ for all } 1\leq i\leq l.\end{eqnarray}
\end{lemma}
\begin{proof}Let $A'_{\gen}$ be the set of elements $\gamma=(\gamma_1,\dotsc, \gamma_l)\in (\Z\cup\{\infty\})^{l}\backslash\{\underline{\infty}\}$ such that $\gamma$ satisfies (\ref{eqn_ci_lambda}). We first show that $A'_{\gen}$ is a multi-semimodule of type $N$. For this, it suffices to show that  $\min(\gamma, \gamma')\in A'_{\gen}$ for all $\gamma, \gamma'\in A'_{\gen}$ as the other conditions are obviously satisfied. Suppose $\gamma=(\gamma_1,\dotsc, \gamma_l)$, $\gamma'=(\gamma'_1,\dotsc, \gamma'_l)$, and $\gamma''=\min(\gamma, \gamma')=(\gamma''_1,\dotsc, \gamma''_l)$. Then $\gamma''_i=\min(\gamma_i, \gamma'_i)$ for $1\leq i\leq l$. Suppose $\gamma''_i=\gamma_i$ (the other case being analogous). By definition
$$c_i(\gamma'')=\sharp\{j|1\leq j\leq l, \min(\gamma_j, \gamma'_j)>\min(\gamma_i, \gamma'_i)\}\leq c_i(\gamma).$$ Then $$\gamma''_i=\gamma_i\in (c_i(\gamma)mn+m\mathbb{N}+n\mathbb{N})\cup \{\infty\}\subseteq (c_i(\gamma'')mn+m\mathbb{N}+n\mathbb{N})\cup\{\infty\}.$$ Hence $\gamma''\in A'_{\gen}$ and $A'_{\gen}$ is a multi-semimodule of type $N$. As all generators of $A_{\gen}$ are contained in $A'_{\gen}$, this implies that $A_{\gen}\subseteq A'_{\gen}$.

Let now $\gamma=(\gamma_1,\dotsc, \gamma_l) \in A'_{\gen}$. Up to permutation, we may assume that the sequence $(\gamma_i)_i$ is decreasing. More precisely, suppose
$$\gamma_1=\cdots=\gamma_{i_1}>\gamma_{i_1+1}=\cdots=\gamma_{i_2}>\cdots>\gamma_{i_{s_0-1}+1}=\cdots=\gamma_{i_{s_0}}.$$  For $1\leq s\leq s_0$, by definition $\gamma_{i_s}\in (i_{s-1}mn+m\mathbb{N}+n\mathbb{N})\cup\{\infty\}$ where $i_0=0$.  Let $$\lambda_s=(\underbrace{\infty, \dotsc, \infty}_{i_{s-1}}, \underbrace{\gamma_{i_s},\dotsc, \gamma_{i_s}}_{l-i_{s-1}})$$ for all $1\leq s\leq s_0$. Then $\lambda_s\in A_{\gen}$ and $\gamma=\min(\lambda_1,\lambda_2, \dotsc,\lambda_{s_0})\in A_{\gen}$. Therefore $A_{\gen}=A'_{\gen}$.
\end{proof}

\begin{prop}\label{lemma_generic_multi-sm} As before assume that $N$ is isoclinic. Suppose that $M\subseteq M_0$ is a Dieudonn\'e lattice with $a(M)=1$ and $M\nsubseteq jM_0$ for all $j\in J(F)$ with $v_\e(det j)>0$. Then $A(M)=A_{\gen}$.
\end{prop}
\begin{proof}By \cite{modpdiv} Lemma 4.8 and Lemma 4.13, each of the generators of $A_{\gen}$ is contained in $A(M)$. Hence $A_{\gen}\subseteq A(M)$. It remains to show $A(M)\subseteq A_{\gen}$. We use induction on $l$. Let $\gamma(\tilde{v})=(\gamma_1,\dotsc,\gamma_l)\in A(M)$ for some $\tilde{v}\in M\backslash\{0\}$. We want to show that $\gamma(\tilde v)\in A_{\gen}$. Up to permutation of the $l$ simple isocrystals inside $N$ (and as $A_{\gen}$ is invariant under this action of $S_l$), we may assume that $\gamma_l$ is maximal among the $\gamma_i$.

If $\gamma_l=\infty$, then
$\tilde{v}\in M\cap N^{\neq l}$ where $N^{\neq l}=\oplus_{i=1}^{l-1} N_{1i}$. Again by \cite{modpdiv}, 4 we have that $\pi^{-\underline{mn}}M\cap N^{\neq l}$ satisfies the assumptions of the proposition for the isocrystal $N^{\neq l}$. Here we use $\underline{mn}=(mn)_i$. By the induction hypothesis and Lemma~\ref{lemma_Agen}, $\gamma_i\in (mn+c'_i(\gamma)mn+m\mathbb{N} +n\mathbb{N})\cup\{\infty\}$ where $c'_i(\gamma):=\sharp\{j|1\leq j\leq l-1, \gamma_j>\gamma_i\}$. As $c_i(\gamma)-1\leq c'_i(\gamma)\leq c_i(\gamma)$ with $c_i(\gamma)$ defined in Lemma~\ref{lemma_Agen}, we have $\gamma_i\in (c_i(\gamma)mn+m\mathbb{N} +n\mathbb{N})\cup\{\infty\}$ and hence $\gamma\in A_{\gen}$.

Let now $\gamma_l\neq \infty$. As $M/(\Phi(M)+V(M))$ is one-dimensional, there is an element $v\in M$ such that $v$ generates $M$ as a Dieudonn\'e lattice. We have $\underline\gamma(v)=\underline 0$. Thus $\gamma_l=\underline\gamma(\tilde v)_l=mx_0+ny_0$ for some $x_0, y_0\in\mathbb{N}$. As before let $\Phi=b\sigma$, $V=\e\Phi^{-1}$. Then there exists
\[\Theta=[a_{x_0, y_0}]\Phi^{x_0}V^{y_0}+\sum_{x, y\in\mathbb{N}\atop mx+ny> \gamma_l}[a_{x, y}]\Phi^x V^y\in O_L[[\Phi, V]]\] with $a_{x,y}\in \overline{k}$ such that $\tilde v-\Theta v\in M\cap N^{\neq l}$ (\cite{modpdiv}, 4). If $\tilde v=\Theta v$, then $\gamma(\tilde v)=\gamma(\Theta v)=(\gamma_l,\dotsc, \gamma_l)\in A_{\gen}$. Otherwise $\tilde v\neq \Theta v$ and $\gamma(\tilde v-\Theta v)\in A_{\gen}$ by the same argument as for the case $\gamma_l=\infty$. Suppose $\gamma(\tilde v-\Theta v)=(\gamma'_1,\dotsc, \gamma'_l)$. For $1\leq i\leq l$, if $\gamma_i=\gamma_l$, then $\gamma'_i\geq \gamma_i$, and otherwise $\gamma'_i=\gamma_i$. Therefore
$\gamma(\tilde v)=\min(\gamma(\Theta v), \gamma(\tilde v-\Theta v))\in A_{\gen}$.
\end{proof}

Let $\tau_0=(b\sigma)^{h}\e^{-m}$ and $\tau_1=\pi=(b\sigma)^a\e^{a'}$ such that $a\in \Z$ is minimal positive such that $a'\in\Z$ with $am+a'h=1$ exists.
\begin{lemma}\label{lemma_description_jLambda_0} A lattice $M$ in the isocrystal $(L^h, b\sigma)$ is of the form $M=jM_0$ with $j\in J_b(F)$ if and only if $M$ is stable under $\tau_0$ and $\tau_1$.
\end{lemma}
\begin{proof} One direction follows as $M_0$ is stable under $\tau_0$ and $\tau_1$, and elements of $J_b(F)$ commute with powers of $\e$ and of $b\sigma$.

For the other direction let as before $e_{ij}$ ($i=1,\dotsc, l$, $j=0,\dotsc,h-1$) be a basis of $M_0$ with $b\sigma(e_{ij})=e_{ij+m}$ (with the usual convention), i.e. $\pi^l(e_{ij})=e_{ij+l}$. We need to find a basis of $M$ with the same properties (as then the base change matrix is in $J_b(F)$). Choose $\tau_0$-invariant elements $f_{i0}$ of $M$ whose classes form a basis of the $\tau_0$-invariant $\bar k$-vector space $M/\tau_1M$. Let $f_{ij}=\tau_1^j(f_{i0})$ for $1\leq j\leq h_1$. Notice that $b\sigma=\tau_1^m\tau_0^c$ and $\tau_1^n=\tau_0^a\e$. Using this one  easily checks that this yields a system of generators with the claimed properties.
\end{proof}

\begin{prop}
Given an isoclinic isocrystal $(N,b\sigma)$ there is a function $f:J_b(F)\rightarrow X_*(T)_{\dom}$ such that $g\in X_{\mu}(b)$ has $a$-invariant 1 if and only if $f_g$ lies in the $J_b(F)$-orbit of $F$.
\end{prop}
\begin{proof} We (may) assume that we are not in the case that $b$ has constant Newton slope 0 or 1 (otherwise there are no elements in $X_{\mu}(b)$ with $a$-invariant 1).

Let $M=gM_0\subset N$ be any lattice. By Lemma \ref{lemma_description_jLambda_0} we see that there is a unique minimal lattice $jM_0$ containing $M$. Notice that this lattice $jM_0$ can be determined from $f_g$. By replacing $g$ with $j^{-1}g$ (which leaves the $J_b(F)$-orbit of $f_g$ invariant) one may assume that $M\subseteq M_0$ and that $M$ is not contained in any lattice $jM_0$ for any $j\in J_b(F)$ with $v_\e(\det j)>0$.

 We first show that all $f_g$ with $a(gM_0)=1$ are in one $J_b(F)$-orbit. Suppose $M=gM_0$ with $a(M)=1$ and assume that $M\subseteq M_0$ and $M\nsubseteq jM_0$ for all $j\in J_b(F)$ with $v_\e(\det j)>0$. Then it is enough to show that $f_g$ is a function not further depending on $M$. By the Cartan decomposition we have
\[J_b(F)=\prod_{z}GL_{l_z}(D_{\lambda_z})=\bigcup_{\underline{u}\in\Z^{\sum_z l_z}\atop u_{z1}\geq\cdots\geq u_{zl_z} \forall z} (\prod_z GL_{l_z}(\O_{D_{\lambda_z}}))\pi^{\underline{u}}(\prod_z GL_{l_z}(\O_{D_{\lambda_z}})).\]
Let $j=k_1\pi^{\underline{u}}k_2$ be the decomposition of some $j\in J_b(F)$. Then in order to compute $f_g(j)$, it is enough to compute $\vol(j^{-1}gM_0\cap \e^iM_0)$ for all $i$. We have
\begin{eqnarray*}
\vol(j^{-1}gM_0\cap \e^iM_0)&=&\vol(k_2^{-1}\pi^{-\underline{u}}k_1^{-1}gM_0\cap \e^iM_0)\\
&=&\vol(\pi^{-\underline{u}}k_1^{-1}gM_0\cap \e^iM_0)\\
&=&\vol(k_1^{-1}gM_0\cap(\pi^{\underline{u}}\e^i)M_0)-v_{\e}(\det \pi^{\underline u}).
\end{eqnarray*}
Notice that $\pi^{\underline{u}}\e^i$ is of the form $\pi^{\underline{u'}}$ for some $\underline{u'}$, thus it is enough to compute the multi-semimodule of $k_1^{-1}gM_0$. However, as $k_1\in J_b(F)\cap K$, this lattice still has $a$-invariant $1$, is contained in $M_0$, and not in $jM_0$ for any $j\in J_b(F)$ with $v_\e(\det j)>0$. Thus this multi-semimodule is the one given by Proposition \ref{lemma_generic_multi-sm}.

Let now $M=gM_0$ with $g\in X_{\mu}(b)$ and $a(M)>1$. By replacing $g$ within its $J_b(F)$-orbit we may again assume that $M\subseteq M_0$ and that $M$ is not contained in any lattice $jM_0$ for any $j\in J_b(F)$ with $v_\e(\det j)>0$. We have to show that the multi-semimodule of our renormalized $M$ is not equal to $A_{\gen}$.

Let $v\in M$ be not contained in any lattice $jM_0$ for any $j\in J_b(F)$ with $v_\e(\det j)>0$ (any sufficiently general element of $M$ will do). Let $M'$ be the Dieudonn\'e lattice generated by $v$.

By our choice of $v$ we have $M'\subset M\subseteq M_0$ with the first containment being a strict one as $a(M)>1$. Further, $M'$ is not contained in any lattice $jM_0$ for any $j\in J_b(F)$ with $v_\e(\det j)>0$. By Proposition \ref{lemma_generic_multi-sm}, $A(M')=A_{\gen}$. By Corollary \ref{koracontained}, $A(M)\neq A(M')$ which finishes the proof.
\end{proof}

\bigskip
\obeylines
Department of Mathematics
Shanghai Key Laboratory of PMMP
East China Normal University
No. 500, Dong Chuan Road
Shanghai, 200241, P.R.China
email: mfchen@math.ecnu.edu.cn

\bigskip
\obeylines
Technische Universit\"at M\"unchen
Zentrum Mathematik -- M11
Boltzmannstr. 3
85748 Garching, Germany
email: viehmann@ma.tum.de

\end{document}